\RequirePackage{fix-cm}
\RequirePackage{etex}

\documentclass[smallcondensed]{svjour3}     % onecolumn (ditto)

\smartqed  % flush right qed marks, e.g. at end of proof
\usepackage{graphicx}
\usepackage{amsmath}
\usepackage{amssymb}
\usepackage{amsfonts}
\usepackage{stmaryrd}
\usepackage{graphicx}
\usepackage{pgfplots}
\usepackage[caption=false]{subfig}
\usepackage{hhline}
\usepackage{multirow}
\usepackage{appendix}
\usepackage{xparse}
\usepackage{mathtools}
\MakeRobust{\eqref}
\usepackage{xspace}
\usepackage{ifthen}
\usepackage{enumitem}
\usepackage{url}
\usepackage{epstopdf}
\usepackage{algorithm,algpseudocode,array}
\usepackage[style=base]{caption}
\usepackage{autonum}
\ifpdf
\DeclareGraphicsExtensions{.eps,.pdf,.png,.jpg}
\else
\DeclareGraphicsExtensions{.eps}
\fi

\makeatletter
\def\@tvsp{\mathchoice{{}\mkern-3mu}{{}\mkern-3mu}{{}\mkern-3mu}{}}
\def\ltrivert{|\@tvsp|\@tvsp|}
\def\rtrivert{|\@tvsp|\@tvsp|}
\makeatother
\makeatletter
\def\@avgsp{\mathchoice{{}\mkern-6mu}{{}\mkern-6mu}{{}\mkern-6mu}{}}
\def\llaverage{\{\@avgsp\{}
\def\rraverage{\}\@avgsp\}}
\makeatother
\NewDocumentCommand{\normDG}{O{\cdot} O{j} O{}}{\ensuremath{\|\ifthenelse{\equal{#1}{}}{\cdot}{#1}\|_{DG,\ifthenelse{\equal{#2}{}}{j}{#2}}\ifthenelse{\equal{#3}{}}{}{^{#3}}}\xspace} 

\NewDocumentCommand{\normL}{O{\cdot} O{2} O{\Om} O{}}{\ensuremath{\|\ifthenelse{\equal{#1}{}}{\cdot}{#1}\|_{L^{\ifthenelse{\equal{#2}{}}{1}{#2}}(\ifthenelse{\equal{#3}{}}{\Om}{#3})}\ifthenelse{\equal{#4}{}}{}{^{#4}}}\xspace}

\NewDocumentCommand{\normH}{O{\cdot} O{1} O{\Om} O{}}{\ensuremath{\|\ifthenelse{\equal{#1}{}}{\cdot}{#1}\|_{H^{\ifthenelse{\equal{#2}{}}{1}{#2}}(\ifthenelse{\equal{#3}{}}{\Om}{#3})}\ifthenelse{\equal{#4}{}}{}{^{#4}}}\xspace}

\NewDocumentCommand{\norms}{O{\cdot} O{s} O{j} O{}}{\ensuremath{\ltrivert\ifthenelse{\equal{#1}{}}{\cdot}{#1}\|_{\ifthenelse{\equal{#2}{}}{s}{#2},\ifthenelse{\equal{#3}{}}{j}{#3}}\ifthenelse{\equal{#4}{}}{}{^{#4}}}\xspace}

\NewDocumentCommand{\Aa}{O{j} O{\cdot} O{\cdot}}{\ensuremath{\mathcal{A}_{\ifthenelse{\equal{#1}{}}{j}{#1}}(\ifthenelse{\equal{#2}{}}{\cdot}{#2},\ifthenelse{\equal{#3}{}}{\cdot}{#3})}\xspace} 

\NewDocumentCommand{\mcal}{O{} O{} O{}}{\ensuremath{\mathcal{#1}\ifthenelse{\equal{#2}{}}{}{_{#2}}\ifthenelse{\equal{#3}{}}{}{^{#3}}}\xspace}

\newcommand{\average}[1]{\ensuremath{\llaverage #1\rraverage}\xspace}  
\newcommand{\jump}[1]{\ensuremath{\llbracket #1\rrbracket}\xspace}

\newcommand{\Om}{\ensuremath{\Omega}\xspace}
\newcommand{\elem}{\ensuremath{\kappa}\xspace}

\newcommand{\h}[1]{\ensuremath{h_{#1}}\xspace}
\newcommand{\hmin}[1]{\ensuremath{h_{#1}}\xspace}

\newtheorem{assumption}[theorem]{Assumption}

\begin{document}

\title{V-cycle multigrid algorithms for discontinuous Galerkin methods on non-nested polytopic meshes
\thanks{This work has been supported by the research grant PolyNuM founded by Fondazione Cariplo and Regione Lombardia, and by the SIR Project n. RBSI14VT0S funded by MIUR.}
}

\titlerunning{V-cycle algorithms for DG methods on non-nested polytopic meshes}

\author{P. F. Antonietti \and G. Pennesi.
}

\institute{P. F. Antonietti  \at
              MOX-Laboratory for Modelling and Scientific Computing, Dipartimento di Matematica, Politecnico di Milano, Piazza Leonardo da Vinci 32, 20133 Milano, Italy.  \\
             Tel.: (+39) 02 2399 4601\\
              \email{paola.antonietti@polimi.it} 
           \and
           G. Pennesi   \at
              MOX-Laboratory for Modelling and Scientific Computing, Dipartimento di Matematica, Politecnico di Milano, Piazza Leonardo da Vinci 32, 20133 Milano, Italy.  \\
           Tel.: (+39) 02 2399 4604\\
           \email{giorgio.pennesi@polimi.it} 
}

%\date{Received: date / Accepted: date}

\maketitle

%%%%%%%%%%
\begin{abstract}
In this paper we analyse the convergence properties of V-cycle multigrid algorithms for the numerical solution of the linear system of equations arising from discontinuous Galerkin discretization of second-order elliptic partial differential equations on polytopal meshes. Here, the sequence of spaces that stands at the basis of the multigrid scheme is possibly non nested and is obtained based on employing agglomeration with possible edge/face coarsening. We prove that the method converges uniformly with respect to the granularity of the grid and the polynomial approximation degree $p$, provided that the number of smoothing steps, which depends on $p$, is chosen sufficiently large.
\end{abstract}
%%%%%%%%%%

%%%%%%%%%%%%%%%
%%     INTRODUCTIOn    %%
%%%%%%%%%%%%%%%
\section{Introduction}
\label{sec:introduction}
The Discontinuous Galerkin (DG) method was introduced in 1973 by Reed and Hill for the discretization of hyperbolic equations \cite{ReedHill73}. Extensions of the method were quickly proposed to deal with elliptic and parabolic problems: some of the most relevant works include Arnold \cite{Ar1982}, Baker \cite{Baker77}, Nitsche \cite{Nitsche} and Wheeler \cite{Wh1978}, whose contributions put the basis for the development of the interior penalty DG methods. 
In the last 40 years the scientific and industrial community has shown an exponentially growing interest in DG methods - see for example \cite{CockKarnShu00,DiPiEr,HestWar,Riviere} for an overview. On one side, the features of DG methods have been naturally enhanced by the recent development of High Performance Computing technologies as well as the growing request for high-order accuracy. In particular, as the discrete polynomial space can be defined locally on each element of the mesh, DG methods feature a high-level of intrinsic parallelism. Moreover, the local conservation properties and the possibility to use meshes with hanging nodes make DG methods interesting also from a practical point of view. 
Recently, it has been shown that DG methods can be extended to computational grids characterized by polytopic elements, cf. Ref. \cite{AnBrMa2008,AnFaRuVe2016,AntGiaHou13,AntGiaHou14,AnHoHuSaVe2017,AnMa17,BaBoCoSu14,Basetal12,BaBoetal12,Canetal14,GiHo2014,LiDaVaYo2014,WiKuMiTaWeDa2013}. In particular, the efficient approach presented in \cite{Canetal14} is based on defining a local polynomial discrete space by making use of the bounding box of each element \cite{GiHo2014_2}: this technique together with a careful choice of the discontinuity penalization parameter permits the use of polytopal elements which can be characterized by faces of arbitrarily small measure and as shown in \cite{CaDoGe2016}, see also \cite{AnHoHuSaVe2017}, possibly by an unbounded number of faces. 

On the other hand, the development of fast solvers and preconditioners for the linear system of equations arising from high-order DG discretization is been developed. A recent strand of the literature has focused on multilevel techniques, including Schwarz domain decomposition methods, cf. Ref. \cite{AntGiaHou14,AnHoSm16}, and two-level and multigrid techniques, cf. Ref. \cite{AnHoHuSaVe2017,AntSarVer15}. The efficiency of those methods is more evident in the case of polygonal grids, because the flexibility of the element shape couples very well with the possibility to easily define agglomerated meshes, which is the key ingredient for the developing of multigrid algorithms. In \cite{AnHoHuSaVe2017} a two-level scheme and W-cycle multigrid method is developed to solve the linear system of equations arising from high-order discretization introduced in \cite{Canetal14}. One iteration of the proposed methods consists of an iterative application of the smoothing Richardson operator and the subspace correction step. In particular, the latter is based on a nested sequence of discrete polynomial spaces where the underlying polytopal grid of each subspace is defined by agglomeration. While being faster than other classical iterative methods, the agglomeration approach presents itself some limitations. When the finest grid is unstructured and characterized by polytopic elements, there is the possibility that its very small edges could be inherited by the coarser levels until the one where the linear system is solved with a direct method. In this case the presence of small faces negatively affects the condition number of the associated matrix: indeed, according to \cite{Canetal14}, the discontinuity penalization parameter is defined locally in each face as the inverse of its measure.

In this paper we aim to overcome this issue by solving the same linear system through a multilevel method characterized by a sequence of non-nested agglomerated meshes in order to make sure that the number of faces of the agglomerates does not blows up as the number of levels of our multigrid method increases. This can be achieved for example based on employing edge-coarsening techniques in the agglomeration procedures. The flexibility in the choice of the computational sub-grids leads to the definition of a non-nested multigrid method characterized by a sequence of non-nested multilevel discrete spaces, cf. Ref. \cite{BrVe1990,Zh1990,ZhZh1997}, and where the discrete bilinear forms are chosen differently on each level, cf. Ref. \cite{GoKa2003,GoKa2003_2,MoZh1995}. The first non-nested multilevel method was introduced by Bank and Dupont in \cite{BaDu1981}; a generalized framework was developed by Bramble, Pasciak and Xu in \cite{BrPaXu1991}, and then widely used in the analysis of non-nested multigrid iterations, cf. Ref. \cite{Br1993,BrKwPa1994,BrPa1992,BrPa1993,BrPa1994,BrZh2001,GoPa2000,ScZh1992,XuCh2001,XuLiCh2002}. The method of \cite{BrPaXu1991}, to whom we will refer as the BPX multigrid framework, is able to generalize also the multigrid framework that we will develop in this paper, but the convergence analysis relies on the assumption that $\Aa[j][I_{j-1}^j u][I_{j-1}^j u] \le \Aa[j-1][u][u]$, which might not be guaranteed in the DG setting, as we will see in Sect.~\ref{sec:stability}. Here $\Aa[j]$ and $\Aa[j-1]$ are two bilinear forms suitably defined on two consecutive levels, and $I_{j-1}^j$ is the prolongation operator whose definition is not trivial, differently from the nested case. For this reason the convergence analysis will be presented based on employing the abstract setting proposed by Duan, Gao, Tan and Zhang in \cite{DuanGaoTanZhang}, which permits to develop a full analysis of V-cycle multigrid methods in a non-nested framework relaxing the hypothesis $\Aa[j][I_{j-1}^j u][I_{j-1}^j u] \le \Aa[j-1][u][u]$. We will prove that our V-cycle scheme with non-nested spaces converges uniformly with respect to the discretization parameters provided that the number of smoothing steps, which depends on the polynomial approximation degree $p$, is chosen sufficiently large. This result extends the theory of \cite{AnHoHuSaVe2017} where W-cycle multigrid methods for high-order DG methods with nested spaces where proposed and analyzed.

The paper is organized as follows. In Sect.~\ref{sec:theory} we introduce the interior penalty DG scheme for the discretization of second-order elliptic problems on general meshes consisting of polygonal/polyhedral elements. In Sect.~\ref{sec:preliminary}, we recall some preliminary analytical results concerning this class of schemes. In Sect.~\ref{sec:bpxvc} we define the multilevel BPX framework for the V-cycle multigrid solver based on non-nested grids, and present the convergence analysis of the algorithm. The main theoretical results are validated through a series of numerical experiments in Sect.~\ref{sec:numerical}. In Sect.~\ref{sec:ASSmoth} we propose an improved version of the algorithm, obtained by choosing a smoothing operator based on a domain decomposition preconditioner.
%%%%%%%%%%

%%%%%%%%%%%%%%%%%%
%%     DG DISCRETIZATION    %%
%%%%%%%%%%%%%%%%%%
\section{Model problem and its DG discretization}
\label{sec:theory}
We consider the weak formulation of the Poisson problem, subject to a homogeneous Dirichlet boundary condition: find $u\in V = H^2(\Omega)\cap H_0^1(\Omega)$ such that
\begin{equation}
\mathcal{A}(u,v) =\int_\Omega \nabla u \cdot \nabla v\ dx =\int_\Omega f v\ dx \, \qquad \forall v\in V,
\label{eq:weak}
\end{equation}
with $\Omega \subset \mathbb{R}^d$, $d = 2,3$, a convex polygonal/polyhedral domain with Lipschitz boundary and $f \in L^2(\Omega)$. The unique solution $u \in V$ of problem \eqref{eq:weak} satisfies 
\begin{equation} 
\normH[u][2][\Om] \le C \normL[f][2][\Om].
\label{eq:elliptic_reg}
\end{equation}

In view of the forthcoming multigrid analysis, let $\{\mcal[T][j]\}_{j=1}^J$ be a sequence of tessellation of the domain \Om, each of which is characterized by disjoint open polytopal elements \elem of diameter \h{\elem}, such that $\overline\Om = \bigcup_{\elem\in\mcal[T][j]}\bar\elem$, $j=1,\dots,J$. The mesh size of $\mcal[T][j]$ is denoted by $\h{j} =\max_{\elem\in\mcal[T][j]}\h{\elem}$. To each $\mcal[T][j]$ we associate the corresponding discontinuous finite element space $V_j$, defined as
\begin{equation}
V_j =\{v\in L^2(\Om):v|_\elem\in \mcal[P][p_j](\elem),\elem\in\mcal[T][j]\},
\end{equation}
where $\mcal[P][p_j](\elem)$ denotes the local space of polynomials of total degree at most $p_j\geq1$ on $\elem\in\mcal[T][j]$. 

\begin{remark}
For the sake of brevity we use the notation $x \lesssim y$ to mean $x \le Cy$, where $C>0$ is a constant independent from the discretization parameters. Similarly we write $x \gtrsim y$ in lieu of $x \ge Cy$, while $x \approx y$ is used if both $x \lesssim y$ and $x \gtrsim y$ hold.  
\end{remark}

A suitable choice of $\{\mcal[T][j]\}_{j=1}^J$ and $\{ V_j \}_{j=1}^J$ leads to the $hp$-multigrid non-nested schemes. This method is based on employing, from one side, a set of non-nested partitions $\{\mcal[T][j]\}_{j=1}^J$, such that the coarse level $\mcal[T][j-1]$ is independent from $\mcal[T][j]$, with the only constrain
\begin{equation}
\label{eq:hjhj1}
\h{j-1} \lesssim\h{j} \leq \h{j-1}\qquad \forall\ j = 2,\dots, J,
\end{equation} 
from the other side we assume that the polynomial degree vary from one level to another such that
\begin{equation}\label{eq:pjpj1}
p_{j-1} \le p_j \lesssim p_{j-1}\qquad \forall\ j = 2,\dots, J.
\end{equation}
Additional assumptions on the grids $\{\mcal[T][j]\}_{j=1}^J$ are outlined in the following paragraph.

\subsection{Grid assumptions}\label{subsection_Grid_assumption}
For any $\mcal[T][j]$, we define the \textit{faces} of the mesh \mcal[T][j], $j=1,\dots,J$, as the intersection of the $(d-1)$-dimensional facets of neighbouring elements. This implies that, for $d=2$, a \textit{face} always consists of a line segment, however for $d=3$, the \textit{faces} of \mcal[T][j] are general shaped polygons. Thereby, we assume that each facets of an element $\elem \in \mcal[T][j]$ may be subdivided into a set of co-planar $(d-1)$-dimensional simplices and we refer to them as \textit{faces}. 
In order to introduce the DG formulation, it is helpful to distinguish between boundary and interior element faces, denoted as $\mathcal{F}_j^B$ and $\mathcal{F}_j^I$, respectively. In particular, we observe that $F \subset \partial \Omega$ for $F \in \mathcal{F}_j^B$, while for any $F \in \mathcal{F}_j^I$ we assume that $F \subset \partial \elem^{\pm}$, where $\elem^{\pm}$ are two adjacent elements in $\mcal[T][j]$. Furthermore, we denoted as $\mathcal{F}_j = \mathcal{F}_j^I \cup \mathcal{F}_j^B$ the set of all mesh faces of $\mcal[T][j]$. With this notation, we assume that the sub-tessellation of element interfaces into $(d-1)$-dimensional simplices is given. Moreover, assume that the following assumptions hold, cf. \cite{CaDoGe2016,Canetal16}.

\begin{assumption}\label{ass1}
For any $j=1,\dots,J$, given $\elem \in \mcal[T][j]$ there exists a set of non-overlapping d-dimensional simplices $T_l \subset \elem$, $l=1,\dots,n_{\elem}$, such that for any face $F \subset \partial \elem$ it holds that $\overline{F} = \partial \overline{\elem} \cap \partial \overline{T_l}$ for some l, it holds $\cup_{l=1}^{n_{\elem}}\overline{T_l} \subset \overline{\elem}$, and the diameter $h_{\elem}$ of $\elem$ can be bounded by 
\begin{equation} 
h_{\elem} \lesssim \frac{d |T_l |}{|F|} \quad \forall\ l=1,\dots,n_{\elem}.
\end{equation}
\end{assumption}

\begin{assumption}\label{ass2}
For any $\elem \in \mcal[T][j],\ j=1,\dots,J$, we assume that $h_{\elem}^d \ge | \elem | \gtrsim h_{\elem}^d$, where $d=2,3$ is the dimension of $\Om$.
\end{assumption}

\begin{assumption}\label{ass3}
Every polytopic element $\elem \in \mcal[T][j],\ j=1,\dots,J$, admits a sub-triangulation into at most $m_{\elem}$ shape-regular simplices $\{ \mathfrak{s}_i \}_{ i=1}^{m_{\elem}}$, for some $m_{\elem} \in \mathbb{N}$, such that $\overline{\elem} = \cup_{i=1}^{m_{\elem}} \overline{\mathfrak{s}_i}$ and 
\begin{equation} 
|\mathfrak{s_i}| \gtrsim |\elem| \quad \forall i=1,\dots,m_{\elem},
\end{equation}
\end{assumption}

\begin{assumption}\label{ass4}
Let $\mcal[T][j]^{\#}= \{ \mathcal{K} \}$, denote a covering of $\Omega$ consisting of shape-regular d dimensional simplices $\mathcal{K}$. We assume that, for any $\elem \in \mcal[T][j]$, there exists $\mathcal{K} \in \mcal[T][j]^{\#}$ such that $\elem \subset \mathcal{K}$ and 
\begin{equation}
\max_{\elem \in \mcal[T][j]} card \bigl\{ \elem' \in \mcal[T][j]: \elem' \cap \mathcal{K} \ne \emptyset , \mathcal{K} \in \mcal[T][j]^{\#} \text{ such that }\elem \subset \mathcal{K} \bigr\} \lesssim 1.
\end{equation}
\end{assumption}

\begin{remark}
Assumption~\ref{ass1} is needed in order to obtain the trace inequalities of Lemma~\ref{lem:inversecont} and Lemma~\ref{lem:inversepoly}. Assumption~\ref{ass2} and~\ref{ass3} are required for the inverse estimates of Lemma~\ref{lem:inverse} and Theorem~\ref{thm:eig}. Assumption~\ref{ass4} guarantees the validity of the approximation result and error estimetes of Lemma~\ref{lem:interpDG} and Theorem~\ref{thm:errors}, respectively.% cf. \cite{CaDoGe2016}.
\end{remark}

\begin{remark}
Assumptions~\ref{ass1} allows to employ polygonal and polyhedral elements possibly characterized by face of degenerating Hausdorff measure as well as unbounded number of faces, cf. \cite{CaDoGe2016}, see also \cite{AnHoHuSaVe2017}.
\end{remark}

\subsection{DG formulation}
In order to introduce the DG discretization of \eqref{eq:weak}, we firstly need to define suitable jump and average operators across the faces $F \in \mcal[F][j]$, $j=1,\ldots,J$. 
Let $\boldsymbol{\tau}$ and $v$ be sufficiently smooth functions. For each internal face $F\in\mcal[F][j]^I$, such that $F \subset \partial \elem^{\pm}$, let $\mathbf{n}^\pm$ be the outward unit normal vector to $\partial \elem^{\pm}$, and let $\boldsymbol{\tau}^\pm$ and $v^\pm$ be the traces of the functions $\boldsymbol{\tau}$ and $v$ on $F$ from $\elem^{\pm}$, respectively. The jump and average operators across $F$ are then defined as follows:
\begin{align}
\jump{\boldsymbol{\tau}} &= \boldsymbol{\tau}^+ \cdot \mathbf{n}^+ + \boldsymbol{\tau}^- \cdot \mathbf{n}^-,\quad& \average{\boldsymbol{\tau}} &= \frac{\boldsymbol{\tau}^+  +\boldsymbol{\tau}^-}{2},\qquad & F\in \mcal[F][j]^I,\\
\jump{v} &= v^+ \mathbf{n}^+ + v^-\mathbf{n}^-,& \average{v} &= \frac{v^+  +v^-}{2},\qquad & F\in \mcal[F][j]^I,\\
\ \ \average{ \boldsymbol{\tau}} &=\boldsymbol{\tau},& \jump{v} &= v\ \mathbf{n},&F\in \mcal[F][j]^B,
\end{align}
cf. \cite{Arnetal01}. With this notation, the bilinear form $\Aa[j]: V_j\times V_j\rightarrow \mathbb{R}$ corresponding to the symmetric interior penalty DG method on the $j$-th level is defined by
\begin{equation}
\Aa[j][u][v] =  \sum_{\elem\in\mcal[T][j]}\int_\elem (\nabla u + \mathcal{R}_j(\jump{u} ) ) \cdot ( \nabla v\ + \mathcal{R}_j(\jump{v}))dx +\sum_{F\in\mcal[F][j]}\int_F\sigma_j\jump{u}\cdot\jump{v}\ ds,
\end{equation}
where $\sigma_j\in L^\infty(\mcal[F][j])$ denotes the interior penalty stabilization function, which is defined by 
\begin{equation}
\sigma_j(x)=\begin{cases}
\displaystyle C_{\sigma}^j \max_{\elem\in\{\elem^+,\elem^-\}} \Big\{ \frac{p_j^2}{h_{\elem}}\Big\},\ 
&x\in F, ~F\in\mcal[F][j]^I,\ F \subset \partial\elem^+\cap\partial\elem^-,\\
\displaystyle C_{\sigma}^j \frac{p^2}{h_{\elem}},\ &
x\in F, ~F\in\mcal[F][j]^B,\ F \subset \partial\elem^+\cap\partial\Omega,\\
\end{cases}
\end{equation}
with $C_{\sigma}^j>0$ independent of $p$, $|F|$ and $|\elem|$, and $\mathcal{R}_j:[L^1(\mcal[F][j])]^d \rightarrow [V_j]^d$ is the lifting operator on the space $V_j$, defined as
\begin{equation}
\int_{\Om} \mathcal{R}_j(\mathbf{q}) \cdot \boldsymbol{\eta} = - \int_{\mcal[F][j]} \mathbf{q} \cdot \average{\boldsymbol{\eta}}\ ds \quad \forall \ \boldsymbol{\eta} \in [V_j]^d. 
\label{eq:lifting_R}
\end{equation}
We refer to \cite{Arnetal01} for more details.
\begin{remark}\label{rmrk:Gj}
Here, the formulation with the lifting operators $\mathcal{R}_j$ allows to introduce the discrete gradient operator $\mathcal{G}_j: V_j \rightarrow [V_j]^d$, defined as 
\begin{equation}\label{eq:Gj}
\mathcal{G}_j(v) = \nabla_j v + \mathcal{R}_j(\jump{v})\quad \forall\ j=1,\dots,J,
\end{equation}
where $\nabla_j$ is the piecewise gradient operator on the space $V_j$. The role of $\mathcal{G}_j$ will be clarified in Sect.~\ref{sec:stability}.
\end{remark}

The goal of this paper is to develop non-nested V-cycle multigrid schemes to solve the following problem posed on the finest level $V_J$: find $u_J\in V_J$ such that
\begin{equation}
\label{eq:DGfem}
\Aa[J][u_J][v_J] = \int_\Om fv_J\ dx\quad\forall v_J\in V_J.
\end{equation}
By fixing a basis for $V_J$, i.e. $V_J = span\{ \phi_J^k\}_{k}$, formulation \eqref{eq:DGfem} results in the following linear system of equations 
\begin{equation}
\mathbf{A}_J \mathbf{u}_J = \mathbf{f}_J,
\end{equation} 
where $\mathbf{u}_J$ is the vector of unknowns.
%%%%%%%%%%

%%%%%%%%%%%%%%%%%%%%
%%     PRELIMINARY RESULTS    %%
%%%%%%%%%%%%%%%%%%%%
\section{Preliminary results}
\label{sec:preliminary}
In this section we recall some preliminary results which form the basis of the convergence analysis presented in the next section.

\begin{lemma}
\label{lem:inversecont}
Assume that the sequence of meshes $\{\mcal[T][j]\}_{j=1}^J$, satisfies Assumption~\ref{ass1} and let $\elem \in \mcal[T][j]$, then the following bound holds
\begin{equation}
\normL[v][2][\partial \elem][2]\lesssim \frac{\epsilon}{h_{\elem}} \normL[v][2][\elem][2] + \frac{h_{\elem}}{\epsilon} | v |_{H^1(\elem)}^2  \quad \forall v \in H^1(\elem),
\end{equation}
where $h_{\elem}$ is the diameter of $\elem$ and $\epsilon>0$ is a positive number.
\end{lemma}
The proof of Lemma~\ref{lem:inversecont} is given in Appendix~\ref{appx:trace}.

\begin{lemma}
\label{lem:inversepoly}
Assume that the sequence of meshes $\{ \mcal[T][j] \}_{j=1}^J$ satisfies Assumption~\ref{ass1} and let $\elem \in \mcal[T][j]$. Then, the following bound holds
\begin{equation}
\normL[v][2][\partial \elem][2]\lesssim \frac{p_j^2}{h_{\elem}} \normL[v][2][\elem][2] \quad \forall v \in \mcal[P][p_j](\elem).
\end{equation}
\end{lemma}
We refer to \cite{CaDoGe2016} for the proof.
\medskip

On each discrete space $\{V_j\}$, $j=1,\dots,J$, we consider the following DG norm:
\begin{equation}\label{eq:DGnorm}
\normDG[w][j][2]=\sum_{\elem\in\mcal[T][j]}\int_\elem |\nabla w|^2\ dx + \sum_{F\in\mcal[F][j]} \int_F \sigma_j|\jump{w}|^2\ ds.
\end{equation}
The well-posed of the DG formulation is established in the following lemma.
\begin{lemma}
\label{lem:contcoerc}
The following continuity and coercivity bounds, respectively, hold
\begin{alignat}{3}
\Aa[j][u][v]&\lesssim \normDG[u][j]\normDG[v][j]\quad &&\forall u,v\in V_j, \\
\Aa[j][u][u]&\gtrsim \normDG[u][j][2]\quad &&\forall u\in V_j,
\end{alignat}
\end{lemma}

Next, we recall the following approximation result, which is an analogous bound presented in \cite[Theorem~5.2]{Canetal14}. 
\begin{lemma}
\label{lem:interpDG}
Let Assumption~\ref{ass4} be satisfied, and let $v \in L^2(\Omega)$ such that, for some $k \ge 0$, $v|_{\elem} \in H^k(\elem)$ for each $\elem \in \mcal[T][j]$. Then there exists a projection operator $\widetilde{\Pi}_j: L^2(\Omega)\rightarrow V_j$ such that
\begin{equation} 
\normH[v-\tilde{\Pi}_jv][q][\Omega] \lesssim  \frac{h_j^{s-q}}{p_j^{k-q}}\normH[v][k][\Omega], \quad for \quad 0 \le q \le k,
\end{equation}
where $s=\min\{p_j+1,k\}$ and $p_j \geq 1.$
\end{lemma}

\noindent The result presented in Lemma~\ref{lem:interpDG} leads to the following error bounds for the underlying interior penalty DG scheme. The error in the energy norm has been proved in~\cite{Canetal14}, see also \cite{CaDoGe2016}. $L^2$-estimates can be found in \cite{AnHoHuSaVe2017}. 
\begin{theorem}
\label{thm:errors}
Assume that Assumptions~\ref{ass1} and~\ref{ass4} hold. We denote by 
$u_j\in V_j$, $j=1,\ldots,J$, the DG solution of problem \eqref{eq:DGfem} posed on level $j$, \emph{i.e.}, 
\begin{equation}
\Aa[j][u_j][v_j] = \int_\Om fv_j\ dx\quad\forall v_j\in V_j.
\end{equation}
If the solution $u$ of \eqref{eq:weak} satisfies $u|_\elem\in H^k(\elem)$, $k \ge 2$, then 
\begin{equation}
\normDG[u-u_j][j] \lesssim \frac{h_j^{(s-1)}}{p_j^{(k-\frac{3}{2})}}\normH[u][k][\Om],\quad 
\normL[u-u_j][2][\Om] \lesssim \frac{\h{j}^{s}}{p_j^{k-1}}\normH[u][k][\Om],
\end{equation}
where $s=\min\{p_j+1,k\}$ and $p_j \geq 1$.
\end{theorem}

\begin{remark}
We point out that the bounds in Theorem~\ref{thm:errors} are optimal in $h$ and suboptimal in $p$ of a factor $p^{\frac{1}{2}}$ and $p$ for the $DG$-norm and the $L^2$-norm, respectively. Optimal error estimates with respect to $p$ can be shown, for example, by using the projector of \cite{GeoSu03} for quadrilateral meshes providing the solution belongs to a suitable augmented Sobolev space. The issue of proving optimal estimates as the ones in \cite{GeoSu03} on polytopic meshes is an open problem and it is under investigation. In the following, we will write:
\begin{equation}
\normDG[u-u_j][j] \lesssim \frac{h_j^{(s-1)}}{p_j^{(k-1-\frac{\mu}{2})}}\normH[u][k][\Om],\quad
\normL[u-u_j][2][\Om] \lesssim \frac{\h{j}^{s}}{p_j^{k-\mu}}\normH[u][k][\Om],
\end{equation}
where $s=\min\{p_j+1,k\}$, $p_j\geq 1$, and $\mu \in \{ 0,1 \}$ for optimal and suboptimal estimates, respectively.
\end{remark}

We also need to introduce an appropriate inverse inequality, cf. \cite{Canetal16}.%,AnHoHuSaVe2017}. 
\begin{lemma}
\label{lem:inverse}
Assume that Assumptions~\ref{ass2} and~\ref{ass3} hold. Then, for any $v\in V_j$, $j=1,\ldots,J$, the following inverse estimate holds
\begin{equation}
\normL[\nabla u][2][\elem][2]\lesssim  p_j^4 h_\elem^{-2}\normL[u][2][\elem][2] \qquad \forall \elem \in \mcal[T][j].
\end{equation}
\end{lemma}

\noindent Thanks to the inverse estimate of Lemma~\ref{lem:inverse}, it is possible to obtain the following upper bound on the maximum eigenvalue of \Aa. We refer to \cite{AntHou} for a similar result on standard grids, and to \cite{AnHoHuSaVe2017} for its extension to polygonal grids.
\begin{theorem}
\label{thm:eig}
Let Assumptions~\ref{ass1},~\ref{ass2} and~\ref{ass3} be satisfied. Moreover, we assume that $h_j = \max_{\elem \in \mcal[T][j]} h_{\elem} \approx h_{\elem}\ \forall \elem \in \mcal[T][j]$, for $j=1,\dots,J$. Then 
\begin{equation}
\Aa[j][u][u]\lesssim \frac{p_j^4}{\hmin{j}^2} \normL[u][2][\Om][2]\quad \forall u\in V_j,\quad j=1,\ldots,J.
\end{equation}
\end{theorem}
%%%%%%%%%%

%%%%%%%%%%%%%%%%%
%%     BPX FRAMEWORK    %%
%%%%%%%%%%%%%%%%%
\section{The BPX-framework for the V-cycle algorithms}
\label{sec:bpxvc}
The analysis presented in this section is based on the general multigrid theoretical framework already employed and developed in \cite{BrPaXu1991} for non-nested spaces and non-inherited bilinear forms. In order to develop a geometric multigrid, the discretization at each level $V_j$ follows the one already presented in \cite{AntSarVer15}, where a W-cycle multigrid method based on nested subspaces is considered. The key ingredient in the construction of our proposed multigrid schemes is the inter-grid transfer operators.

Firstly, we introduce the operators $A_j:V_j\rightarrow V_j$, defined as
\begin{equation}
\label{eq:defAj}
(A_ju,v)=\Aa[j][u][v]\quad\forall u,v\in V_j, \quad j=1,\dots,J,
\end{equation}
and we denote as $\Lambda_j\in\mathbb{R}$ the maximum eigenvalue of $A_j$ $\forall\ j=2,\dots,J$.
Moreover, let $\textnormal{Id}_j$ be the identity operator on level $V_j$.  The smoothing scheme, which is chosen to be the Richardson iteration, is then characterized by the following operators:
\begin{equation}
B_j = \Lambda_j\textnormal{Id}_j \quad j=2,\dots,J.
\end{equation}

The prolongation operator connecting the coarser space $V_{j-1}$ to the finer space $V_j$ is denoted by $I_{j-1}^j$. Since the two spaces are non-nested, i.e. $V_{j-1} \not\subset V_j$, it cannot be chosen as the ''natural injection operator''. The most natural way to define the prolongation operator is the $L^2$-projection, i.e. $I_{j-1}^j : V_{j-1}\rightarrow V_j$
\begin{equation}\label{eq:Def_I}
(I_{j-1}^j v_H, w_h)_{L^2(\Om)} = ( v_H, w_h)_{L^2(\Om)} \ \forall w_h \in V_j,
\end{equation}

\noindent The restriction operator $I_{j}^{j-1} : V_{j}\rightarrow V_{j-1}$ is defined as the adjoint of $I_{j-1}^j$ with respect to the $L^2(\Omega)$-inner product, i.e.,
\begin{equation}
(I_j^{j-1} w_h,v_H)_{L^2(\Omega)} = (w_h,I_{j-1}^j v_H)_{L^2(\Omega)}\ \ \forall v_H \in V_{j-1}.
\end{equation}

\noindent For our analysis, we also need to introduce the operator $P_j^{j-1}:V_j\rightarrow V_{j-1}$ such that:
\begin{equation}
\Aa[j-1][P_j^{j-1} w_h][v_H] = \Aa[j][w_h][I_{j-1}^j v_H]  \quad \forall v_H \in V_{j-1}, w_h \in V_j.
\end{equation}

According with \eqref{eq:defAj}, problem~\eqref{eq:DGfem} can be written in the following equivalent form: find $u_J \in V_J$ such that
\begin{equation}
\label{eq:system}
A_J u_J = f_J,
\end{equation}
where $f_J \in V_J$ is defined as $(f_J,v)_{L^2(\Om)} = \int_\Om f v\ dx\ \forall v\in V_J$. Given an initial guess $u_0 \in V_J$, and choosing parameters $m_1,m_2 \in \mathbb{N}$, the multigrid V-cycle iteration algorithm for the approximation of $u_J$ is outlined in Algorithm~\ref{alg:VcycleMethod}. In particular, $\mathsf{MG}_\mathcal{V} (J,f_J,u_k,m_1,m_2)$ represents the approximate solution obtained after one iteration of our non-nested V-cycle scheme, which is defined by induction: if we consider the general problem of finding $z \in V_j$ such that
\begin{equation}
\label{eq:systemVj}
A_j z = g,
\end{equation}
with $j\in \{2,\dots,J\}$ and $g \in L^2(\Om)$, then $\mathsf{MG}_\mathcal{V} (j,g,z_0,m_1,m_2)$ represents the approximate solution of~\eqref{eq:systemVj} obtained after one iteration of the non-nested V-cycle scheme with initial guess $z_0 \in V_j$ and $m_1,\ m_2$ number of pre-smoothing and post-smoothing steps, respectively. The recursive procedure is outlined in Algorithm~\ref{alg:multilevel}, where we also observe that on the level $j=1$ the problem is solved by using a direct method.

\begin{algorithm}[t!]
\caption{Multigrid V-cycle iteration for the solution of problem~\eqref{eq:system}}
\label{alg:VcycleMethod}
\begin{algorithmic}
\State Initialize $u_0 \in V_J$;
\For{$k=0,1,\dots$} 
\State $u_{k+1}=\mathsf{MG}_\mathcal{V} (J,f_J,u_k,m_1,m_2);$ 
\State $u_k = u_{k+1}$;
\EndFor
\end{algorithmic}
\end{algorithm}

\begin{algorithm}[t!]
\caption{One iteration of the Multigrid V-cycle scheme on the level $j \ge 2$}
\label{alg:multilevel}
\begin{algorithmic}
\If{j=1}
\State $\mathsf{MG}_\mathcal{V} (1,g,z_0,m_1,m_2)=A_1^{-1}g.$
\Else
\State \underline{\textit{Pre-smoothing}}:
\For{$i=1,\dots,m_1$} \State $z^{(i)}=z^{(i-1)}+B_j^{-1}(g-A_jz^{(i-1)});$ \EndFor \vspace{0.3cm}
\State $\underline{\textit{Coarse grid correction}}$:
\State $r_{j-1} = I_j^{j-1}(g-A_jz^{(m_1)})$;
\State $e_{j-1} = \mathsf{MG}_\mathcal{V} (j-1,r_{j-1},0,m_1,m_2)$;
\State $z^{(m_1+1)}=z^{(m_1)}+I_{j-1}^je_{j-1}$;\vspace{0.3cm}
\State $\underline{\textit{Post-smoothing}}$:
\For{$i=m_1+2,\dots,m_1+m_2+1$} \State $z^{(i)}=z^{(i-1)}+B_j^{-1}(g-A_jz^{(i-1)});$ \EndFor \vspace{0.3cm}
\State $\mathsf{MG}_\mathcal{V} (j,g,z_0,m_1,m_2)=z^{(m_1+m_2+1)}.$
\EndIf
\end{algorithmic}
\end{algorithm}

%%%%%%%%%%%%%%%%%%%%%
%%     CONVERGENCE ANALYSIS    %%
%%%%%%%%%%%%%%%%%%%%%
\subsection{Convergence analysis}
\label{sec:2lvlconv}

We first define the following norms on each discrete space $V_j$
\begin{equation}
\ltrivert v\rtrivert_{s,j}=\sqrt{(A_j^sv,v)_{L^2(\Om)}}\qquad \forall\ s\in\mathbb{R},\ v\in V_j,\quad j=1, \dots,J.
\end{equation}
To analyze the convergence of the algorithm, for any $j=2,\dots,J$ we set $G_j = \textnormal{Id}_j - B_j^{-1}A_j$ and let $G_j^*$ be its adjoint respect to $\Aa[j][][]$. Following \cite{DuanGaoTanZhang}, we make three standard assumptions in order to prove the convergence of Algorithm~\ref{alg:VcycleMethod}:
\begin{enumerate}[label=\textbf{A.\arabic*}]
\addtolength{\itemindent}{0.3cm}
\item \label{ass:stability} Stability estimate: $\exists\ C_Q>0$ such that
\begin{equation}
\ltrivert (\textnormal{Id}_j - I_{j-1}^j P_j^{j-1})u_h \rtrivert_{1,j} \leq C_Q \ltrivert  u_h\rtrivert_{1,j}   \qquad \forall u_h \in V_j, \quad j=2,\dots,J;
\end{equation}

\item \label{ass:approximation} Regularity-approximation property: $\exists\ C_1>0$ such that
\begin{equation}
\bigl| \Aa[j][(\textnormal{Id}_j - I_{j-1}^j P_j^{j-1})u_h][u_h] \bigr| \leq C_1 \frac{\ltrivert u_h\rtrivert_{2,j}^2}{\Lambda_j} \qquad \forall u_h \in V_j, \quad j=2,\dots,J,
\end{equation}
where $\Lambda_j = \max \lambda_i(A_j)$ ;

\item \label{ass:smoother} Smoothing property: $\exists\ C_R>0$ such that
\begin{equation}
\frac{\normL[u_h][2][\Om]}{\Lambda_j} \leq C_R \bigl( \mathcal{R} u_h, u_h \bigr)  \qquad \forall u_h \in V_j, \quad j=2,\dots,J,
\end{equation}
where $\mathcal{R} = \bigl( \textnormal{Id}_j - G_j^* G_j \bigr) A_j^{-1}$.
\end{enumerate}

\noindent The convergence analysis of the V-cycle method is described by the following theorem that gives an estimate for the error propagation operator, which is defined as
\begin{equation}
\begin{cases} 
\mathbb{E}_{1,m_1,m_2}^{\mathsf{V}}v &= 0,\qquad \qquad \qquad \qquad \qquad \qquad \qquad \qquad \quad \quad \ \ \ \qquad \qquad j=1,\\
\mathbb{E}_{j,m_1,m_2}^{\mathsf{V}}v  &= (G_J^*)^{m_2}(\textnormal{Id}_j - I_{j-1}^j P_j^{j-1} + I_{j-1}^j \mathbb{E}_{j-1,m_1,m_2}^{\mathsf{V}} P_j^{j-1} )G_j^{m_1} v,\ j >1.
\end{cases}
\end{equation}

\begin{theorem}
\label{thm:convergence}
If Assumptions~\ref{ass:stability}, \ref{ass:approximation} and \ref{ass:smoother} hold, then
\begin{equation}
\bigl| \Aa[j][ \mathbb{E}_{j,m,m}^{\mathsf{V}}u ][u] \bigr| \leq \delta_j \Aa[j][u][u] \qquad \forall u \in V_j, \quad j=2,\dots,J
\end{equation}
where $\delta_j = \frac{C_1 C_R}{m - C_1 C_R}<1,$ provided that $m > 2C_1C_R$.
\end{theorem}
We refer to \cite{DuanGaoTanZhang} for the proof of Theorem~\ref{thm:convergence} in an abstract setting. In the following, we prove the validity of Assumptions \ref{ass:stability}, \ref{ass:approximation} and \ref{ass:smoother} for the algorithm presented in this section. We start with a two-level approach, i.e. $J=2$, so we will consider the two-level method for the solution of \eqref{eq:DGfem}, based on two spaces $V_{J-1} \not\subset V_J$. The generalization to the V-cycle method will be given at the end of this section.

%%%%%%%%%%%%%%%%
%%     ASSUMPTION A.1    %%
%%%%%%%%%%%%%%%%
\subsection{Verification of Assumption~\ref{ass:stability}}
\label{sec:stability}
In order to verify Assumption \ref{ass:stability} for the two-level method we first show a stability result of the prolongation operator $I_{J-1}^J$. In the following, we also consider the $L^2$-projection operator on the space $V_J$ defined as
\begin{equation}
Q_J : L^2(\Om) \rightarrow V_J \text{, such that } (Q_Ju,v_J)_{L^2(\Omega)} = (u,v_J)_{L^2(\Omega)} \quad \forall v_J \in V_J.
\end{equation} 
\begin{remark}\label{rmrk:eqQJandIJ}
From the definition of $I_{J-1}^J$ given in \eqref{eq:Def_I}, it holds $I_{J-1}^J = Q_J |_{V_{J-1}}$.
\end{remark}

\noindent Moreover, we need the following approximation result which shows that any $v_j \in V_j,\ j=J-1,J,$ can be approximated by an $H^1$-function, see \cite{AnHoSm16}. Let $\mathcal{G}_j$ be the discrete gradient operator~\eqref{eq:Gj} introduced in Remark~\ref{rmrk:Gj}, and consider the following problem: $\forall v_j \in V_j$, find $\mathcal{H}(v_j) \in H^1_0(\Om)$ such that
\begin{equation}
\int_{\Om} \nabla \mathcal{H}(v_j) \cdot \nabla w\ dx = \int_{\Om} \mathcal{G}_j(v_j) \cdot \nabla w\ dx \quad \forall w \in H^1_0(\Om).
\label{eq:Hdef}
\end{equation}

\noindent It is shown in \cite{AnHoSm16} that $\mathcal{H}(v_j)$ possesses good approximation properties in terms of providing an $H^1$-conforming approximant of the discontinuous function $v_j$:
\begin{theorem}
Let $\Om$ be a bounded convex polygonal/polyhedral domain in $\mathbb{R}^d$, $d=2,3$. Given $v_j \in V_{j}$, we write $\mathcal{H}(v_j) \in H^1_0(\Om)$ to be the approximation defined in \eqref{eq:Hdef}. Then, the following approximation and stability results hold:
\begin{equation}\label{eq:Hbound}
\normL[v_j - \mathcal{H}(v_j)][2][\Om] \lesssim \frac{h_j}{p_j} \normL[\sigma_j^{\frac{1}{2}} \jump{v_j}][2][\mathcal{F}_h], \qquad | \mathcal{H}(v_j) |_{H^1(\Om)} \lesssim \normDG[v_j][j].
\end{equation} 
\end{theorem}

We make use of the previous result in order to show the following stability result of the prolongation operator: 
\begin{lemma}\label{lemm:C_stab}
There exists a positive constant $\mathsf{C}_{\mathsf{stab}}$, independent of the mesh size such that
\begin{align}
\normDG[I_{J-1}^Jv_H][J] &\leq \mathsf{C}_{\mathsf{stab}}(p_J)\ \normDG[v_H][J-1] \quad \forall v_H \in V_{J-1},
\end{align}
here $\mathsf{C}_{\mathsf{stab}}(p_J) \approx p_J$.
\end{lemma}

\begin{proof}
Let $v_H \in V_{J-1}$, by the definition of the DG-norm~\eqref{eq:DGnorm}, we need to estimate:
\begin{align}
\normDG[I_{J-1}^J v_H][J]^2 & = \normL[\nabla_J (I_{J-1}^Jv_H)][2][\mathcal{T}_J]^2 + \normL[ \sigma_J^{\frac{1}{2}} |\jump{I_{J-1}^J v_H}|][2][\mathcal{F}_J]^2
\label{eq:normIvDG}.
\end{align}
We next bound each of the two terms on the right hand side. For the first one let be $\mathcal{H}_H = \mathcal{H}(v_H)$ defined as in \eqref{eq:Hdef}. Then:
\begin{align}\label{eq:sec_term}
\|  \nabla_J  (I_{J-1}^J v_H) \|_{L^2(\mathcal{T}_J)}^2 & \leq 
 \normL[\nabla_J(I_{J-1}^J v_H - \widetilde{\Pi}_{J}(\mathcal{H}_H) )][2][\mathcal{T}_J]^2  \\& +\normL[\nabla_J(\mathcal{H}_H - \widetilde{\Pi}_{J}(\mathcal{H}_H) )][2][\mathcal{T}_J]^2 + |\mathcal{H}_H|_{H^1(\Om)}^2,
\end{align}
where we have added and subtracted the terms $\nabla_J \widetilde{\Pi}_{J}(\mathcal{H}_H) )$ and $\nabla \mathcal{H}_H$. The second term of the right hand above side can be estimated using the interpolation bounds of Lemma~\ref{lem:interpDG}, the Poincar\'e inequality for $\mathcal{H}_H \in H^1_0(\Omega)$ and the second bound of \eqref{eq:Hbound}:
\begin{equation}
\normL[\nabla_J(\mathcal{H}_H - \widetilde{\Pi}_{J}(\mathcal{H}_H) )][2][\mathcal{T}_J]^2  \lesssim | \mathcal{H}_H |_{H^1(\Omega)}^2 \lesssim  \normDG[v_H][J-1]^2.
\end{equation}
In order to estimate the first term on the right hand side in \eqref{eq:sec_term} we observe that, since $I_{J-1}^J v_H - \widetilde{\Pi}_{J}(\mathcal{H}_H) \in V_J$, it is possible to make use of the inverse inequality of Lemma~\ref{lem:inverse}, that leads to the following bound:
\begin{equation}\label{eq:UseTrace}
 \| \nabla_J(  I_{J-1}^J v_H -  \widetilde{\Pi}_{J}(\mathcal{H}_H) )\|_{L^2(\mathcal{T}_J)}^2 \lesssim p_J^4 h_J^{-2} \normL[I_{J-1}^J v_H -  \widetilde{\Pi}_{J}(\mathcal{H}_H)][2][\mathcal{T}_J][2].
\end{equation}
By adding and subtracting $\mathcal{H}_H$ to $\normL[I_{J-1}^J v_H -  \widetilde{\Pi}_{J}(\mathcal{H}_H)][2][\mathcal{T}_J][2]$ we obtain
\begin{equation}\label{eq:UseAddSubHH}
\normL[I_{J-1}^J v_H -  \widetilde{\Pi}_{J}(\mathcal{H}_H)][2][\mathcal{T}_J][2] \lesssim \normL[I_{J-1}^J v_H -  \mathcal{H}_H][2][\mathcal{T}_J][2]  + \normL[\mathcal{H}_H -  \widetilde{\Pi}_{J}(\mathcal{H}_H)][2][\mathcal{T}_J][2].
\end{equation}
Using Lemma~\ref{lem:interpDG} and the Poincar\'e inequality we have
\begin{equation}
\normL[\mathcal{H}_H -  \widetilde{\Pi}_{J}(\mathcal{H}_H)][2][\mathcal{T}_J][2] \lesssim \frac{h_J^2}{p_J^2} \| \mathcal{H}_H \|_{H^1(\Omega)}^2  \lesssim \frac{h_J^2}{p_J^2} \normDG[v_H][J-1]^2,
\end{equation}
whereas the term $\normL[I_{J-1}^J v_H -  \mathcal{H}_H][2][\mathcal{T}_J][2]$ can be estimate as follow:
\begin{equation}
 \| I_{J-1}^J v_H - \mathcal{H}_H \|_{L^2(\mathcal{T}_J)}^2  \lesssim  \normL[I_{J-1}^J v_H -  Q_J(\mathcal{H}_H)][2][\mathcal{T}_J][2] +  \normL[\mathcal{H}_H -  Q_{J}(\mathcal{H}_H)][2][\mathcal{T}_J][2] 
 \end{equation}
Using Remark~\ref{rmrk:eqQJandIJ}, the continuity of $Q_J$ with respect to the $L^2$-norm, Lemma~\ref{lem:interpDG} and \eqref{eq:Hbound} we have
\begin{align}
\| I_{J-1}^J v_H - & \mathcal{H}_H \|_{L^2(\mathcal{T}_J)}^2  \lesssim \normL[Q_J(v_H - \mathcal{H}_H)][2][\mathcal{T}_J][2]  + \normL[\mathcal{H}_H -  Q_{J}(\mathcal{H}_H)][2][\mathcal{T}_J][2]  \\
& \lesssim  \normL[v_H - \mathcal{H}_H][2][\mathcal{T}_J][2] + \normL[\mathcal{H}_H - \widetilde{\Pi}_J(\mathcal{H}_H) ][2][\mathcal{T}_J][2]   \\
& \lesssim \frac{h_J^2}{p_J^2} \normL[\sigma_J^{\frac{1}{2}} |\jump{v_H}][2][\mathcal{F}_J]^2 +  \frac{h_J^2}{p_J^2} \normH[\mathcal{H}_H][1][\Om]^2  \lesssim \frac{h_J^2}{p_J^2} \normDG[v_H][J-1]^2.
\end{align}
Thanks to the previous estimates and inequalities~\eqref{eq:UseAddSubHH}, it holds
\begin{equation}\label{eq:TermL2}
 \| I_{J-1}^J v_H - \widetilde{\Pi}_{J}(\mathcal{H}_H)  \|_{L^2(\mathcal{T}_J)}^2 \lesssim \frac{h_J^2}{p_J^2} \normDG[v_H][J-1]^2,
\end{equation} 
the previous estimate, together with \eqref{eq:UseTrace}, \eqref{eq:sec_term} and the bound $|\mathcal{H}_H|_{\Om}^2 \lesssim \normDG[v_H][J-1]^2$ leads to 
\begin{equation}\label{eq:grad_estimate}
\| \nabla_J (I_{J-1}^J v_H) \|_{L^2(\mathcal{T}_J)}^2 \lesssim p_J^2\  \normDG[v_H][J-1]^2.
\end{equation}

Next we bound the second term on the right hand side in \eqref{eq:normIvDG}. By the definition of the jump term and remembering that $\jump{\mathcal{H}_H}=0\ \forall F \in \mcal[F][J]$ since $\mathcal{H}_H \in H_0^1(\Omega)$, it holds
\begin{align}\label{eq:traceIvH}
\| \sigma_J^{\frac{1}{2}} & \jump{I_{J-1}^J v_H} \|_{L^2(\mathcal{F}_J)}^2 \lesssim \\ & \sum_{\elem \in \mathcal{T}_J} \frac{p_J^2}{h_{\elem}} \Bigl( \normL[ I_{J-1}^J v_H - \widetilde{\Pi}_{J}( \mathcal{H}_H)][2][\partial \elem]^2 + \normL[ \widetilde{\Pi}_{J}( \mathcal{H}_H) - \mathcal{H}_H ][2][\partial \elem]^2  \Bigr),
\end{align}
where we also used the definition of $\sigma_J$. Now, we first observe that we could use the trace inequality of Lemma~\ref{lem:inversepoly} in order to obtain
\begin{equation}\label{eq:tracefirst}
\normL[ I_{J-1}^J v_H - \widetilde{\Pi}_{J}( \mathcal{H}_H)][2][\partial \elem]^2 \lesssim \frac{p_J^2}{h_J} \normL[ I_{J-1}^J v_H - \widetilde{\Pi}_{J}( \mathcal{H}_H)][2][\elem]^2.
\end{equation}
To bound the second term on the right hand side in \eqref{eq:traceIvH} we make use of the continuous trace inequality on polygons of Lemma~\ref{lem:inversecont} with $\epsilon = p_J$, the approximation property of Lemma~\ref{lem:interpDG} and the Poincar\'e inequality:
\begin{align}
\| \widetilde{\Pi}_{J}(  \mathcal{H}_H) - \mathcal{H}_H \|_{L^2(\partial \elem)}^2 & \lesssim \frac{p_J}{h_{J}} \normL[\widetilde{\Pi}_{J}( \mathcal{H}_H) - \mathcal{H}_H][2][\elem][2] + \frac{h_{J}}{p_J} | \widetilde{\Pi}_{J}( \mathcal{H}_H) - \mathcal{H}_H |_{H^1(\elem)}^2  \\
& \lesssim  \frac{p_J}{h_{J}} \frac{h_J^2}{p_J^2} \| \mathcal{H}_H \|_{L^2(\elem)}^2 + \frac{h_J}{p_J} \| \mathcal{H}_H \|_{H^1(\elem)}^2  \lesssim \frac{h_J}{p_J} | \mathcal{H}_H |_{H^1(\elem)}^2.
\end{align}
From the previous inequality and the bound \eqref{eq:tracefirst}, \eqref{eq:traceIvH} becomes: 
\begin{align}
\| \sigma_J^{\frac{1}{2}} \jump{I_{J-1}^J v_H} \|_{L^2(\mathcal{F}_J)}^2  & \lesssim  \frac{p_J^4}{h_J^2} \normL[ I_{J-1}^J v_H - \widetilde{\Pi}_{J}( \mathcal{H}_H)][2][\mathcal{T}_J]^2 +  p_J | \mathcal{H}_H |_{H^1(\Omega)}^2 \\& \lesssim p_J^2  \normDG[v_H][J-1]^2,
\end{align}
where we also used inequality~\eqref{eq:TermL2}. This estimate together with \eqref{eq:grad_estimate} lead to
\begin{equation}
\normDG[I_{J-1}^Jv_H][J] \leq \mathsf{C}_{\mathsf{stab}} (p_J)\ \normDG[v_H][J-1] \quad \forall v_H \in V_{J-1}.
\end{equation}
where $\mathsf{C}_{\mathsf{stab}}(p_J) \approx p_J$.
\end{proof}

We can use the previous result in order to prove that Assumption \ref{ass:stability} holds. We first observe that also the operator $P_J^{J-1}$ satisfies a similar stability estimate as the one of $I_{J-1}^J$, that is
\begin{align}
\| P_J^{J-1}& v_h \|_{DG,J-1}^2  \lesssim \Aa[J-1][P_J^{J-1} v_h][P_J^{J-1} v_h] = \Aa[J][v_h][I_{J-1}^J P_J^{J-1} v_h] \\ & \lesssim \normDG[v_h][J] \normDG[I_{J-1}^J P_J^{J-1} v_h][J] \lesssim \mathsf{C}_{\mathsf{stab}} (p_J)\ \normDG[v_h][J] \normDG[P_J^{J-1} v_h][J],
\end{align}
from which it follows 
\begin{equation}
\normDG[P_J^{J-1} v_h][J-1] \lesssim \mathsf{C}_{\mathsf{stab}} (p_J)\ \normDG[v_h][J]. 
\end{equation}

\begin{proposition}
Assumption~\ref{ass:stability} holds with $C_Q \approx p_J^2$. 
\end{proposition}
\begin{proof}
Let $v_H \in V_{J-1},$ making use of Lemma~\ref{lem:contcoerc} we have 
\begin{equation} 
\Aa[J][I_{J-1}^J v_H][I_{J-1}^J v_H]  \lesssim \normDG[I_{J-1}^J v_H][J]^2 \lesssim p_J^2\  \normDG[v_H][J-1]^2 \lesssim p_J^2\ \Aa[J-1][v_H][v_H],
\end{equation}
and similarly it holds 
\begin{equation}\label{eq:AP}
\Aa[J-1][P_J^{J-1} u_h][P_J^{J-1} u_h] \lesssim p_J^2\ \Aa[J][u_h][u_h]\quad \forall u_h \in V_J.
\end{equation} 
Let $u_h \in V_{J}$ and fix $v_H = P_J^{J-1} u_h $, then the following inequality holds:
\begin{equation}\label{eq:AIP}
\Aa[J][I_{J-1}^J P_J^{J-1} u_h][I_{J-1}^J P_J^{J-1} u_h] \lesssim p_J^2\ \Aa[J-1][P_J^{J-1} u_h][P_J^{J-1} u_h].
\end{equation}
By adding and subtracting $u_h$ to both arguments of $\Aa[J]$ on the left hand side of \eqref{eq:AIP}, and using \eqref{eq:AP} we obtain

\begin{equation} 
\underbrace{ \mathcal{A}_J((\textnormal{Id}_J - I_{J-1}^J P_J^{J-1}) u_h, (\textnormal{Id}_J - I_{J-1}^J P_J^{J-1}) u_h) }_{= \ltrivert (\textnormal{Id}_J - I_{J-1}^J P_J^{J-1})u_h\rtrivert_{1,J}^2} \lesssim  \underbrace{ \Bigl( p_J^2 \bigl( p_J^2 - 2 \bigr) + 1 \Bigr)}_{\le p_J^4} \Aa[J][u_h][u_h],
\end{equation}
that concludes the proof.
\end{proof}

%%%%%%%%%%%%%%%%
%%     ASSUMPTION A.2   %%
%%%%%%%%%%%%%%%%
\subsection{Verification of Assumption~\ref{ass:approximation}}

In order to show the validity of Assumption \ref{ass:approximation} we need the following standard approximation result, which is proved in Appendix~\ref{appx:app_prop}.

\begin{lemma}\label{lem:QuasiStability}
Let Assumptions~\ref{ass1} - \ref{ass4} hold. Then
\begin{equation}\label{eq:QuasiApprox}
\normL[(\textnormal{Id}_J - I_{J-1}^J P_J^{J-1})v_J][2][\Om] \lesssim \frac{h_J^2}{p_J^{2-\mu}}   \ltrivert v_J \rtrivert_{2,J} \quad \forall v_J \in V_J.
\end{equation}
\end{lemma}
Thanks to Lemma~\ref{lem:QuasiStability}, it is possible to show the following theorem:
\begin{theorem}
The regularity-approximation property \ref{ass:approximation} holds with $C_1 \approx p_J^{2+\mu}$.
\end{theorem}

\begin{proof} 
Theorem~\ref{thm:eig} gives the following bound of the maximum eigenvalue of $A_J$: $\Lambda_J \lesssim  \frac{p_J^4}{\hmin{J}^2}.$
Using Lemma~\ref{lem:QuasiStability}, the above bound on $\Lambda_J$, and the symmetry of $\Aa[J][][]$ we have, for all $v \in V_J$:
\begin{align} 
\Aa[J][(\textnormal{Id}_J - I_{J-1}^J P_J^{J-1})v][v]  & \leq \ltrivert v \rtrivert_{2,J}  \ltrivert (\textnormal{Id}_J - I_{J-1}^J P_J^{J-1})v \rtrivert_{0,J}  \lesssim \frac{h_J^2}{p_J^{2-\mu}}  \ltrivert v \rtrivert_{2,J}^2\\& \lesssim p_J^{2+\mu} \frac{\ltrivert v \rtrivert_{2,J}^2}{\Lambda_J}.
\end{align} 
that concludes the proof.
\end{proof}

%%%%%%%%%%%%%%%%
%%     ASSUMPTION A.3    %%
%%%%%%%%%%%%%%%%
\subsection{Verification of Assumption~\ref{ass:smoother}}
\begin{proposition}
Assumption \ref{ass:smoother} holds with $C_R=1$.
\end{proposition}
\begin{proof}
We have:
\begin{equation}
\mathcal{R} = \bigl( \textnormal{Id}_J - G_J^* G_J \bigr) A_J^{-1} = \Bigl( \frac{2}{\Lambda_J} A_J - \frac{1}{\Lambda_J^2} A_J A_J \Bigr) A_J^{-1} = \frac{1}{\Lambda_J} \Bigl( \textnormal{Id}_J + \Bigl( \textnormal{Id}_J - \frac{1}{\Lambda_J} A_J \Bigr) \Bigr),
\end{equation}
and so
\begin{equation}
\bigl( \mathcal{R} u,u \bigr)_{L^2(\Omega)} = \frac{\normL[u_h][2][\Om]}{\Lambda_J} + \Bigl( \Bigl( \textnormal{Id}_J - \frac{1}{\Lambda_J} A_J \Bigr)u, u \Bigr)_{L^2(\Omega)}.
\end{equation}
We now prove that $\Bigl( \textnormal{Id}_J - \frac{1}{\Lambda_J} A_J \Bigr)$ is a positive definite operator. By contradiction, let us suppose that there exists a function $\overline{u} \in V_J$, $\overline{u} \ne 0$, such that $\Bigl( \Bigl( \textnormal{Id}_J - \frac{1}{\Lambda_J} A_J \Bigr)\overline{u}, \overline{u} \Bigr)_{L^2(\Omega)} < 0$, then 
\begin{equation} \label{eq:assurdo}
\Lambda_J (\overline{u}, \overline{u})_{L^2(\Omega)} < \Aa[J][\overline{u}][\overline{u}],
\end{equation}
by Lemma~\ref{lem:contcoerc} and the symmetry of the bilinear form $\Aa[J][][]$, the eigenfunctions $\{ \phi^{J}_{k} \}_{k=1}^{N_J}$ satisfy 
\begin{equation} \Aa[J][\phi^{J}_{k}][v] = \lambda^{J}_{k} (\phi^{J}_{k}, v)_{L^2(\Omega)} \qquad \forall v \in V_J, \end{equation} 
where $0 < \lambda^{J}_{1} \leq \lambda^{J}_{2} \leq \dots \leq \lambda^{J}_{N_J} = \Lambda_J$. The set of eigenfunctions is an orthonormal basis for the space $V_J$, i.e. $(\phi^{J}_{i},\phi^{J}_{j})_{L^2(\Omega)} = \delta_{ij}$, and they satisfy $\Aa[J][\phi^{J}_{i}][\phi^{J}_{j}] = \lambda^{J}_{i} \delta_{ij}$,
where $\delta_{ij}$ is the Kronecker symbol. Since $\{ \phi^{J}_{k} \}_{k=1}^{N_J}$ is a basis of the space $V_J$, we can write $\overline{u} = \sum_{k=1}^{N_J} c_k \phi^{J}_{k}$, so that \eqref{eq:assurdo} becomes
\begin{equation}
\Lambda_J \sum_{i,j=1}^{N_J}  c_j(\phi^{J}_{j}, \phi^{J}_{i})_{L^2(\Omega)} c_i < \sum_{i,j=1}^{N_J} c_j \Aa[J][\phi^{J}_{j}][\phi^{J}_{i}] c_i = \sum_{i,j=1}^{N_J} c_j \lambda^{J}_{i} (\phi^{J}_{i},\phi^{J}_{j})_{L^2(\Omega)} c_i,
\end{equation}

\begin{equation}
\Rightarrow\quad \Lambda_J \sum_{i=1}^{N_J}  c_i^2 < \sum_{i,j=1}^{N_J} c_i^2 \lambda^{J}_{i},
\end{equation}
which is a contradiction. We then deduce that $\Bigl( \textnormal{Id}_J - \frac{1}{\Lambda_J} A_J \Bigr)$ is a positive definite operator.
\end{proof}

\begin{remark}
We observe that, as we need to satisfy the condition $m > 2C_1C_R$ of Theorem~\ref{thm:convergence}, we can guarantee the convergence of the method choosing the number $m$ of smoothing steps such that $m \gtrsim p_J^{2+\mu}$, which is in agreement to what proved for W-cycle algorithms in \cite{AntSarVer15} and \cite{AnHoHuSaVe2017} on nested grids.
\end{remark}

\begin{remark}
The analysis of this section can be generalized to the full V-cycle algorithm with $J>2$ as follows: Assumption~\ref{ass:smoother} is verified with $C_R=1$ also on the arbitrary levels $j,j-1$, because each level $j$ satisfies Assumption~\ref{ass:smoother} with constant $C_R^j=1$. Assumptions~\ref{ass:approximation} and \ref{ass:stability} are satisfied with $C_1 = \max_{j}\{ C_1^j \}$ and $C_Q = \max_{j}\{ C_Q^j \}$, respectively, where $C_1^j$ and $C_Q^j$ are the same as the ones defined in the previous analysis but on the level $j$.
\end{remark}

%%%%%%%%%%%%%%%%%%%%%%%
%%%%% NUMERICAL RESUTS %%%%%%
%%%%%%%%%%%%%%%%%%%%%%%
\section{Numerical results}
\label{sec:numerical}

\begin{figure}[t]
\centering
\begin{tabular}{lcccc}
 & Set 1 & Set 2 & Set 3 & Set 4\\
Level 4 & \subfloat{
	  \raisebox{-.45\height}{\includegraphics[width=0.17\textwidth]{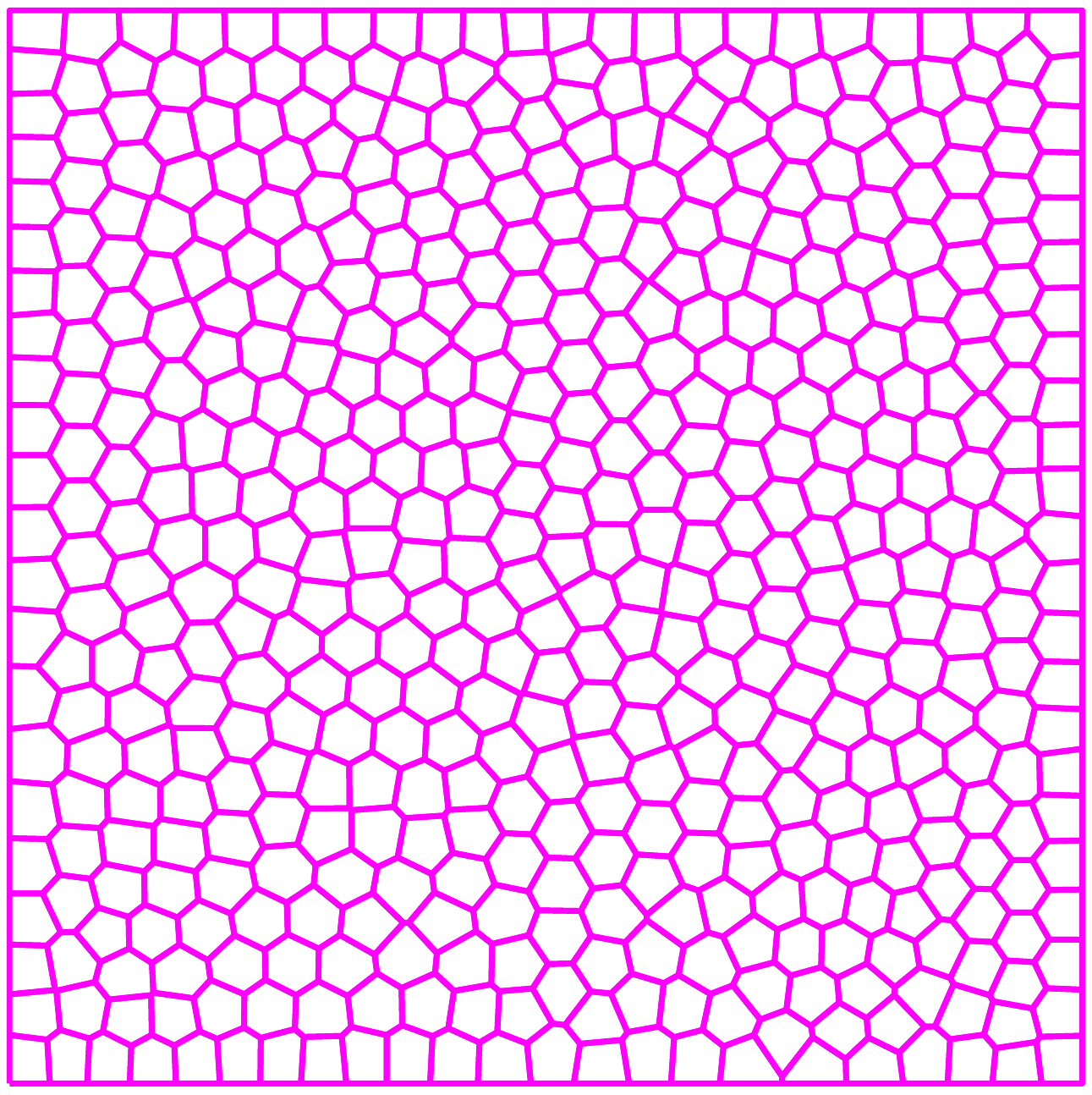}}}  & \subfloat{
          \raisebox{-.45\height}{\includegraphics[width=0.17\textwidth]{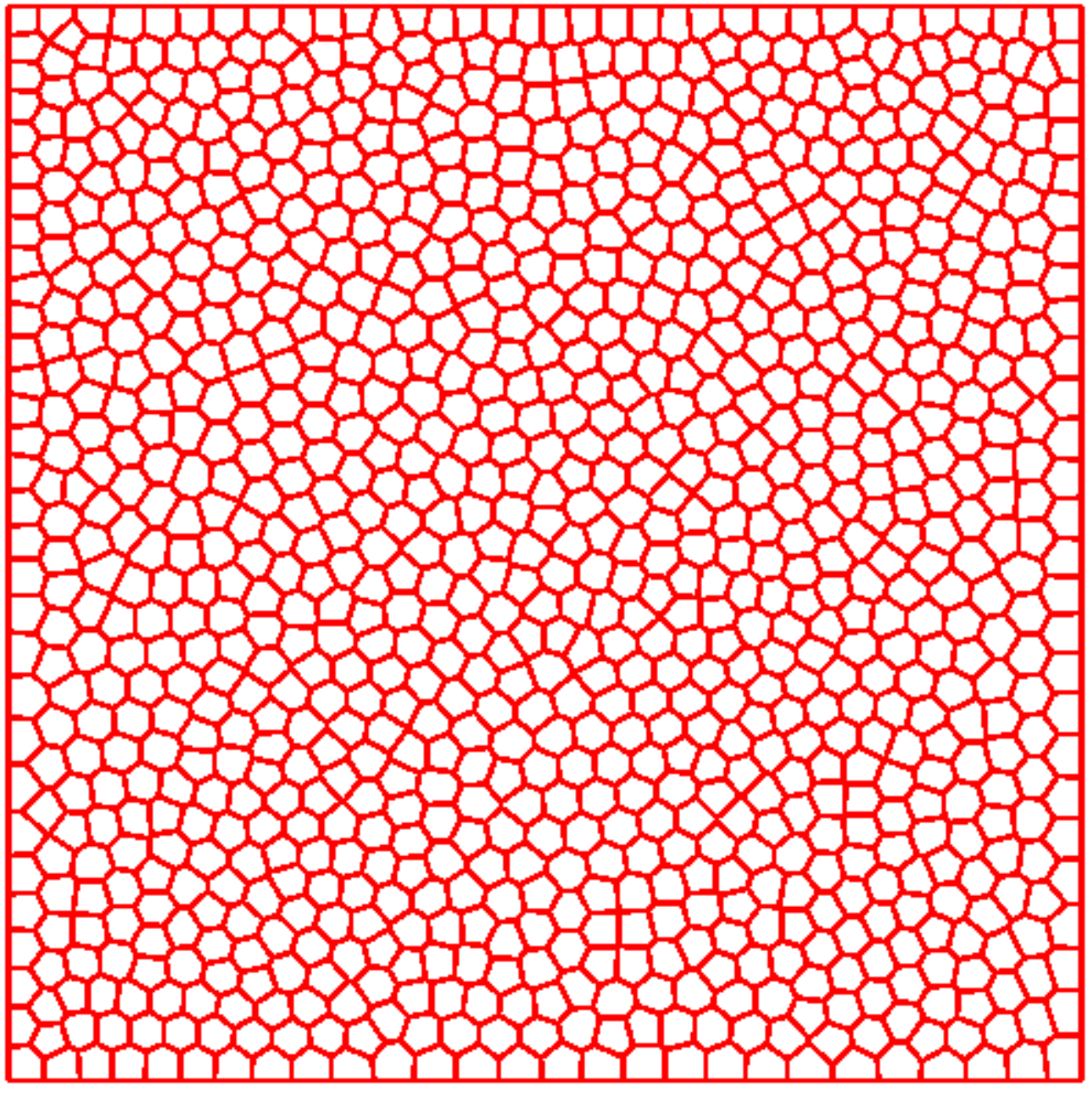}}} & \subfloat{
          \raisebox{-.45\height}{\includegraphics[width=0.17\textwidth]{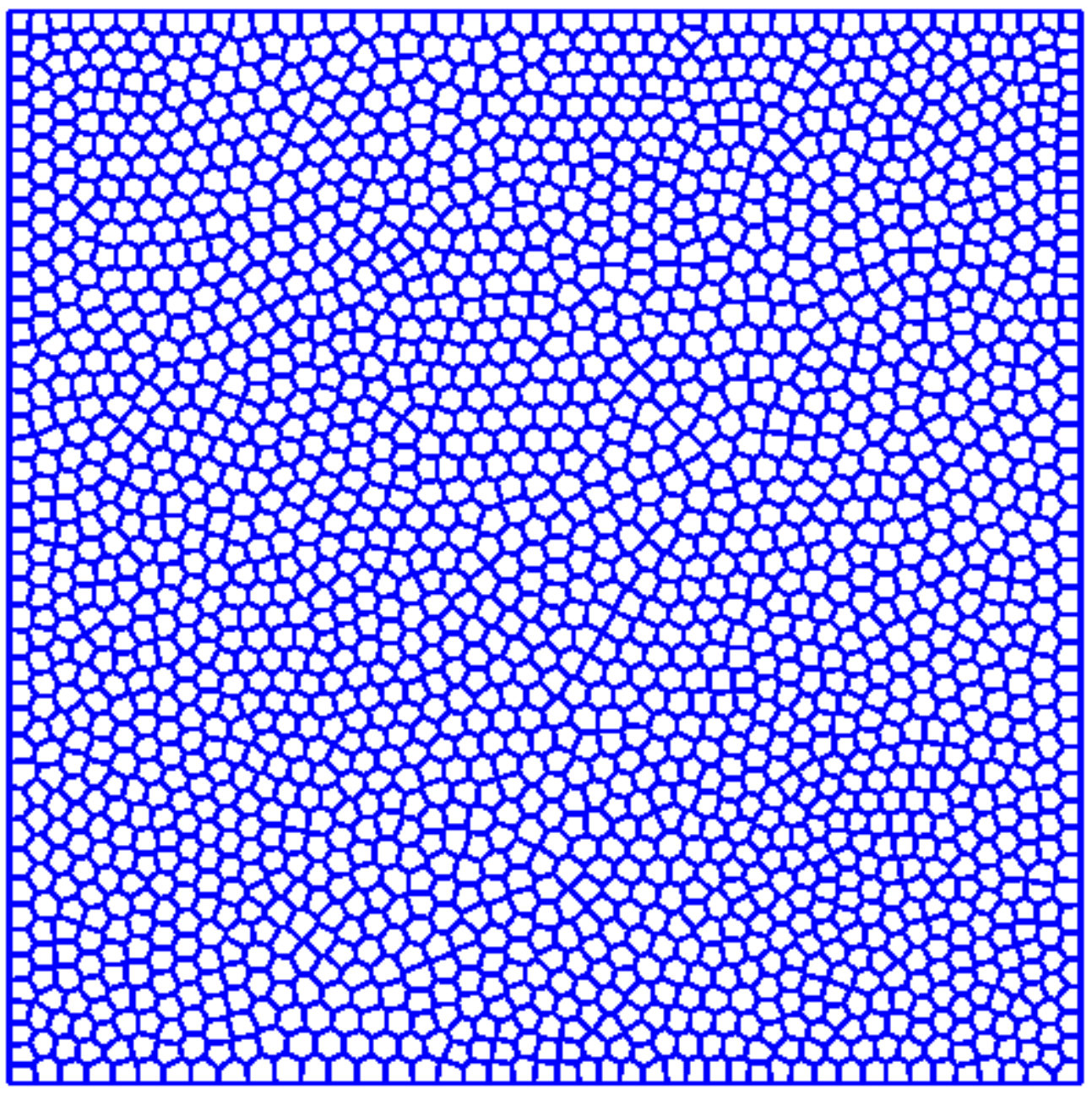}}} & \subfloat{
          \raisebox{-.45\height}{\includegraphics[width=0.17\textwidth]{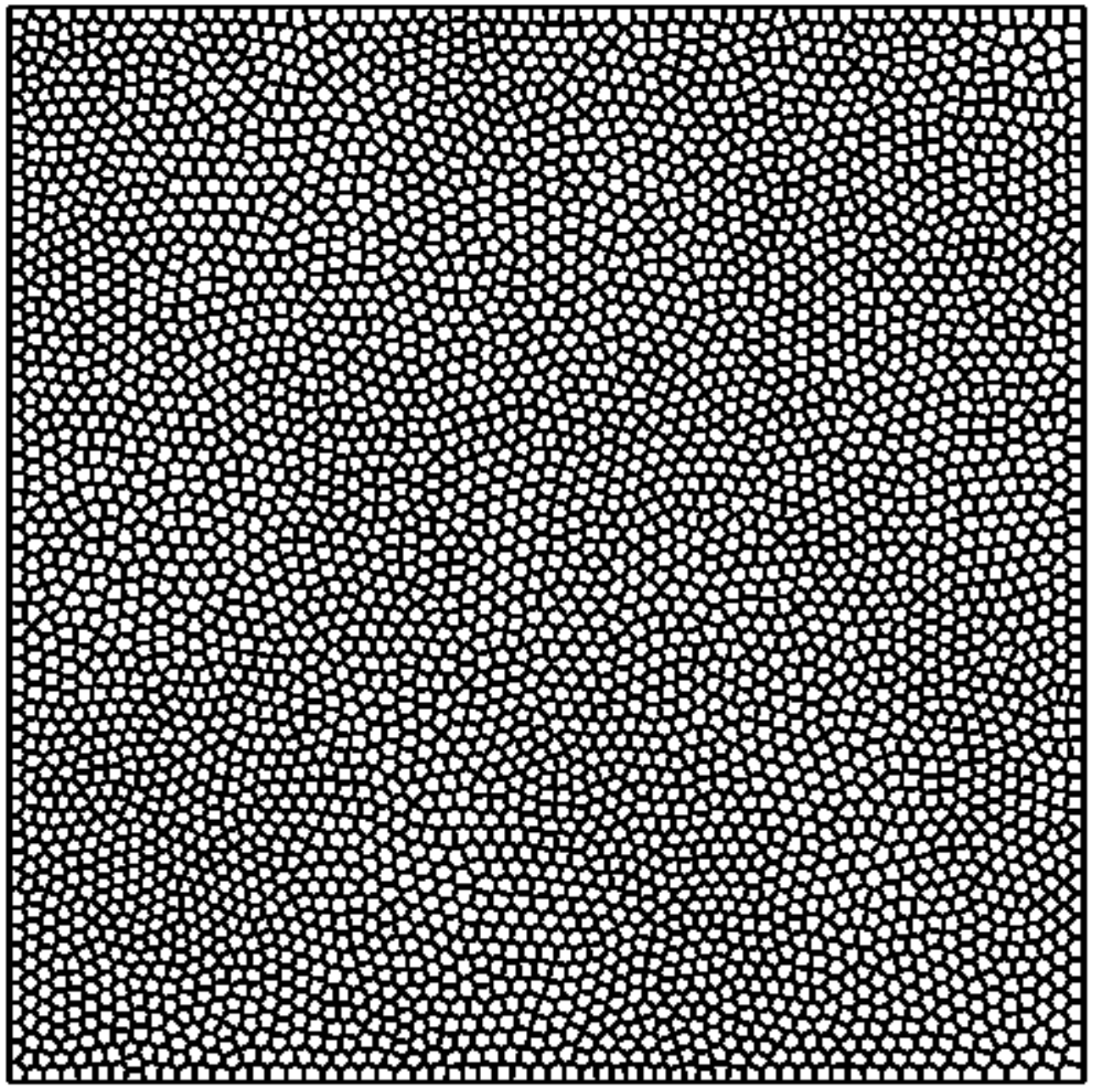}}}\\
Level 3 & \subfloat{
          \raisebox{-.45\height}{\includegraphics[width=0.17\textwidth]{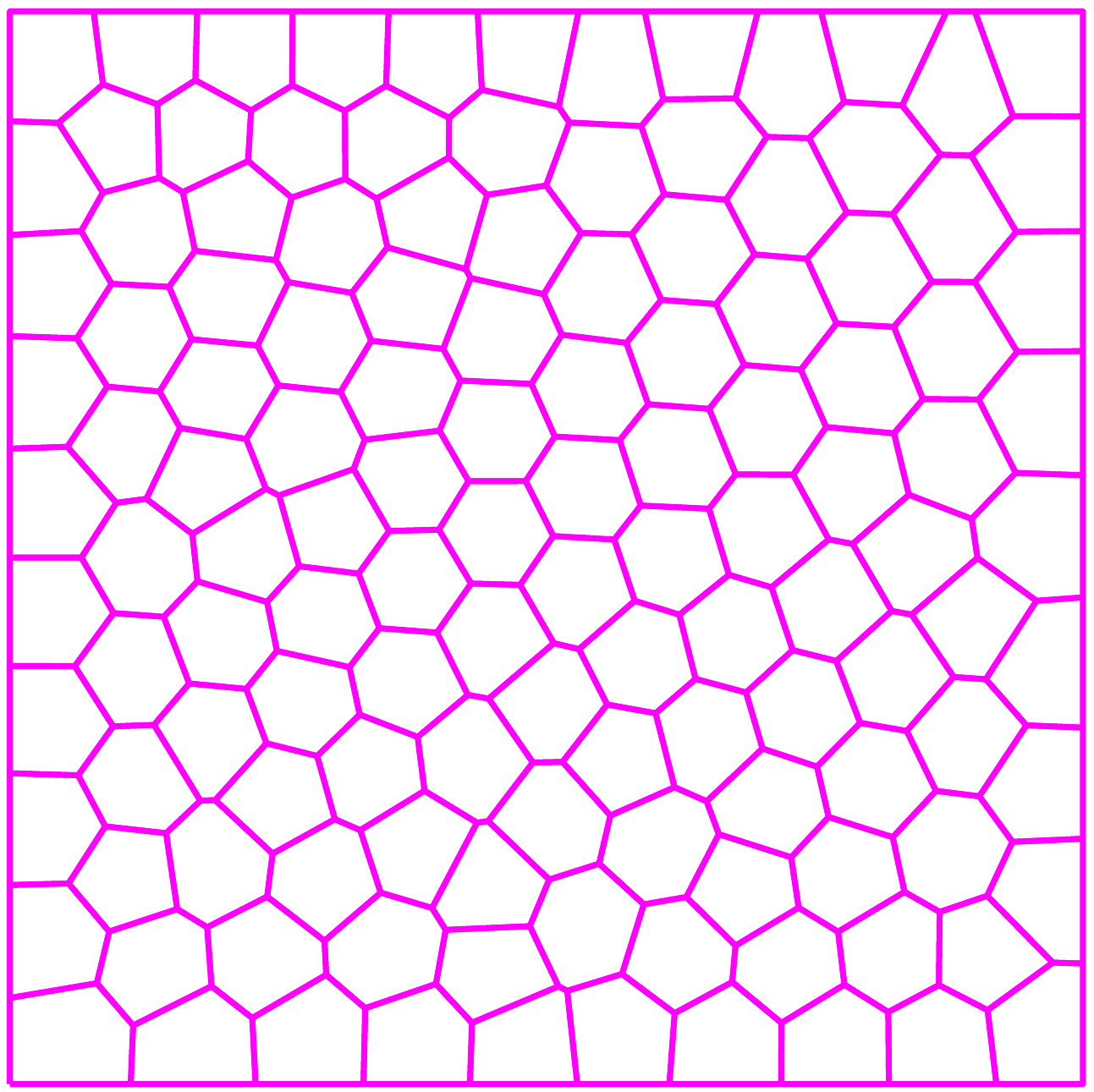}}}  & \subfloat{
          \raisebox{-.45\height}{\includegraphics[width=0.17\textwidth]{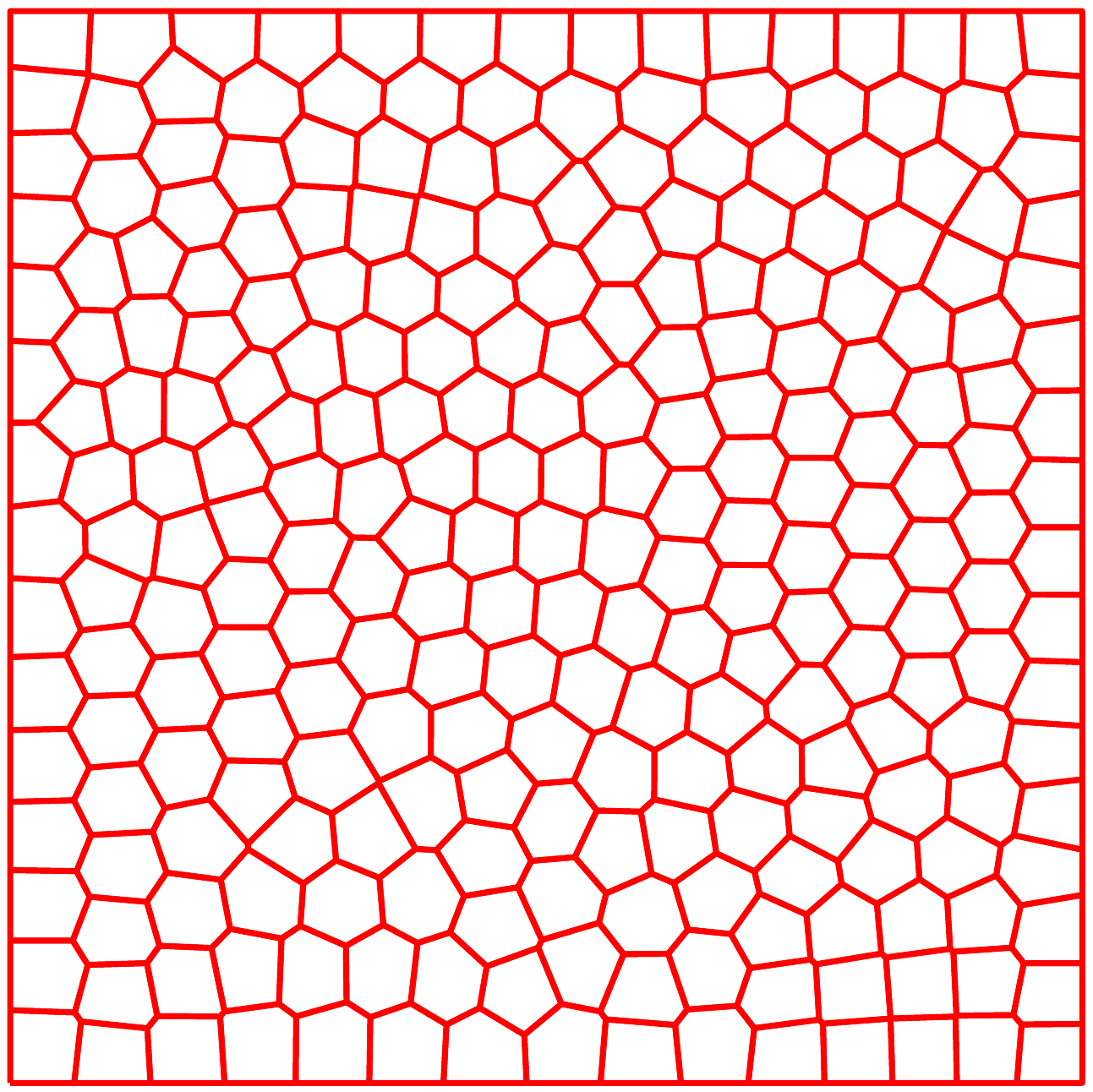}}} & \subfloat{
          \raisebox{-.45\height}{\includegraphics[width=0.17\textwidth]{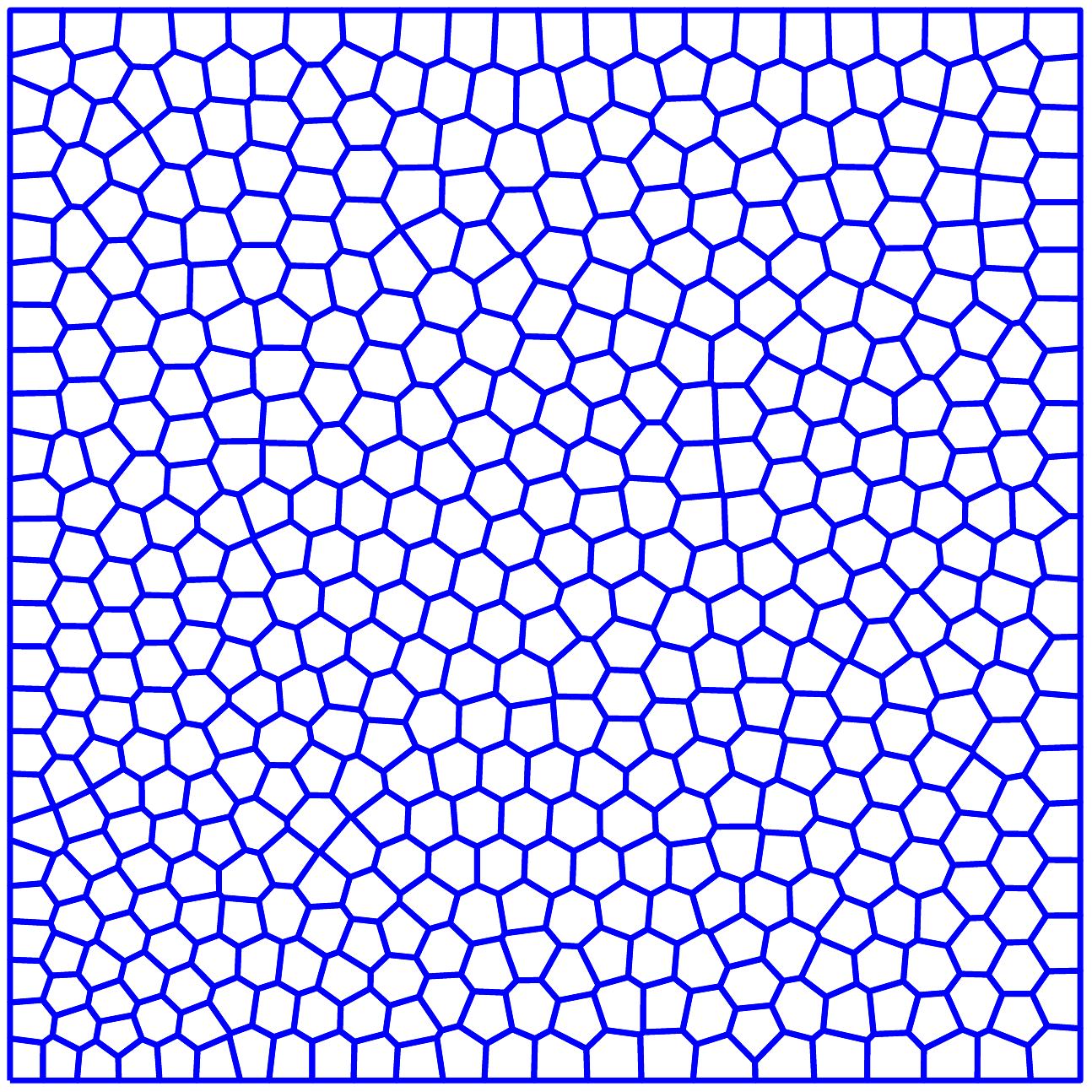}}} & \subfloat{
          \raisebox{-.45\height}{\includegraphics[width=0.17\textwidth]{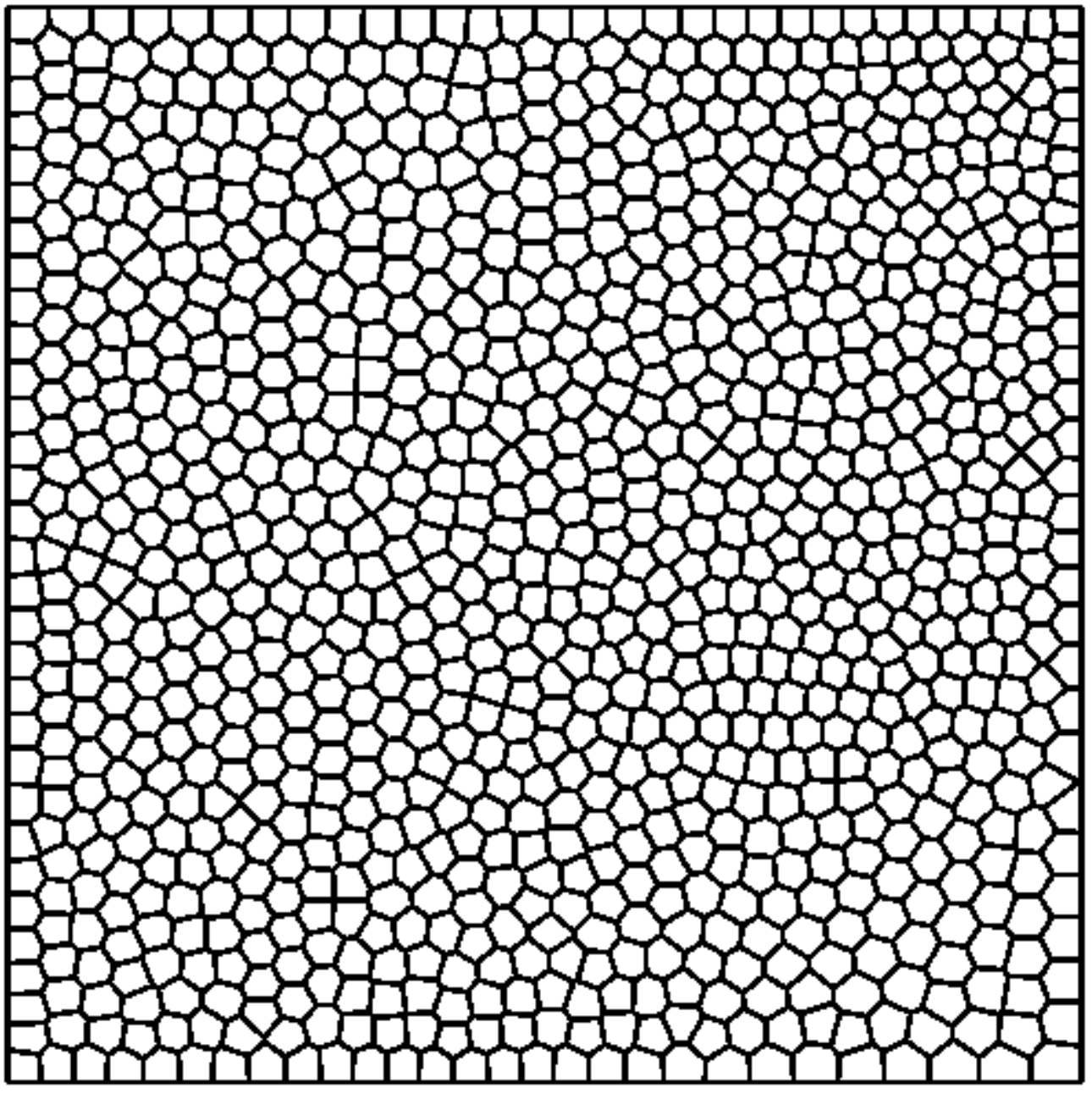}}}\\
Level 2 & \subfloat{
          \raisebox{-.45\height}{\includegraphics[width=0.17\textwidth]{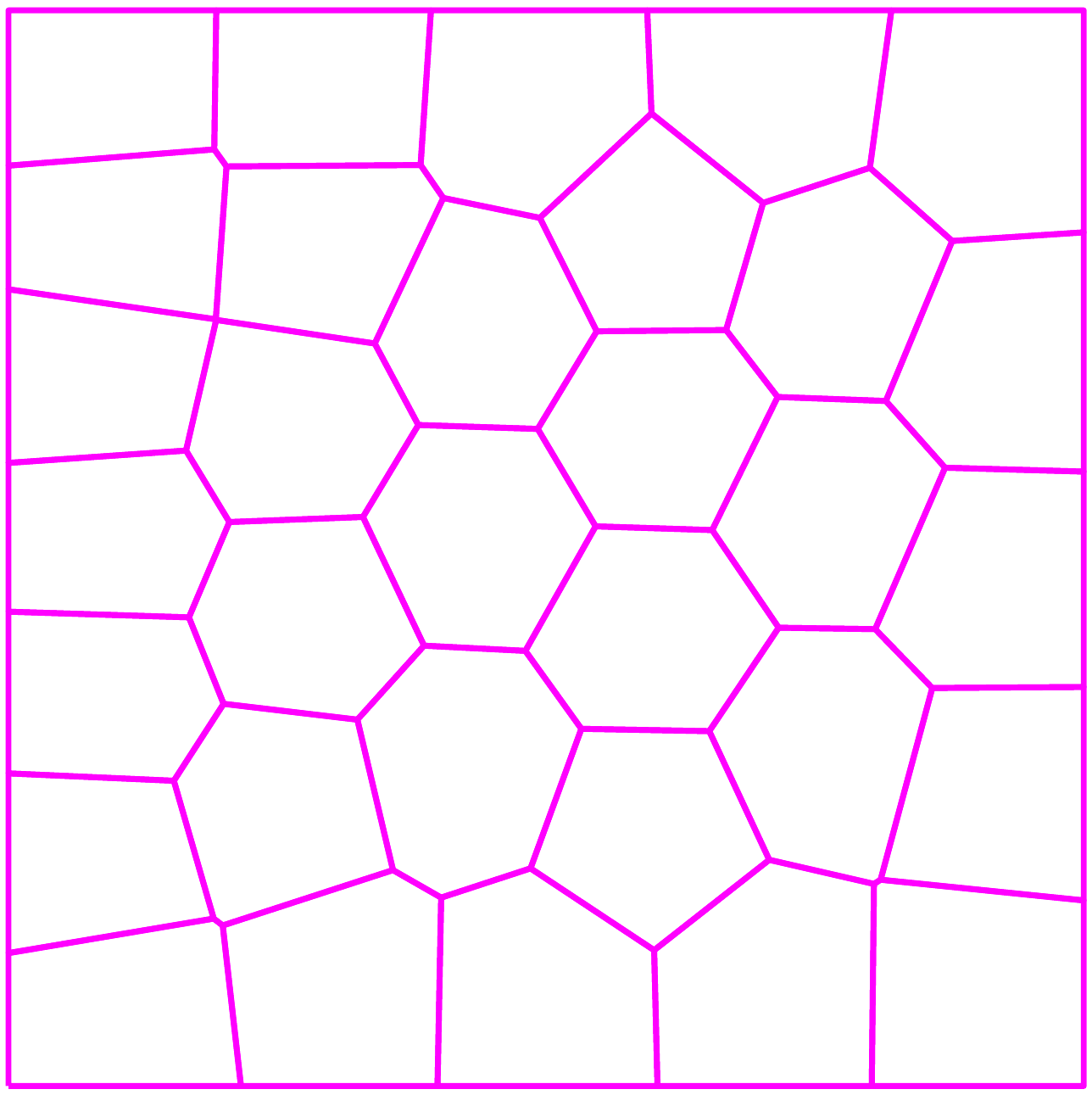}}}  & \subfloat{
          \raisebox{-.45\height}{\includegraphics[width=0.17\textwidth]{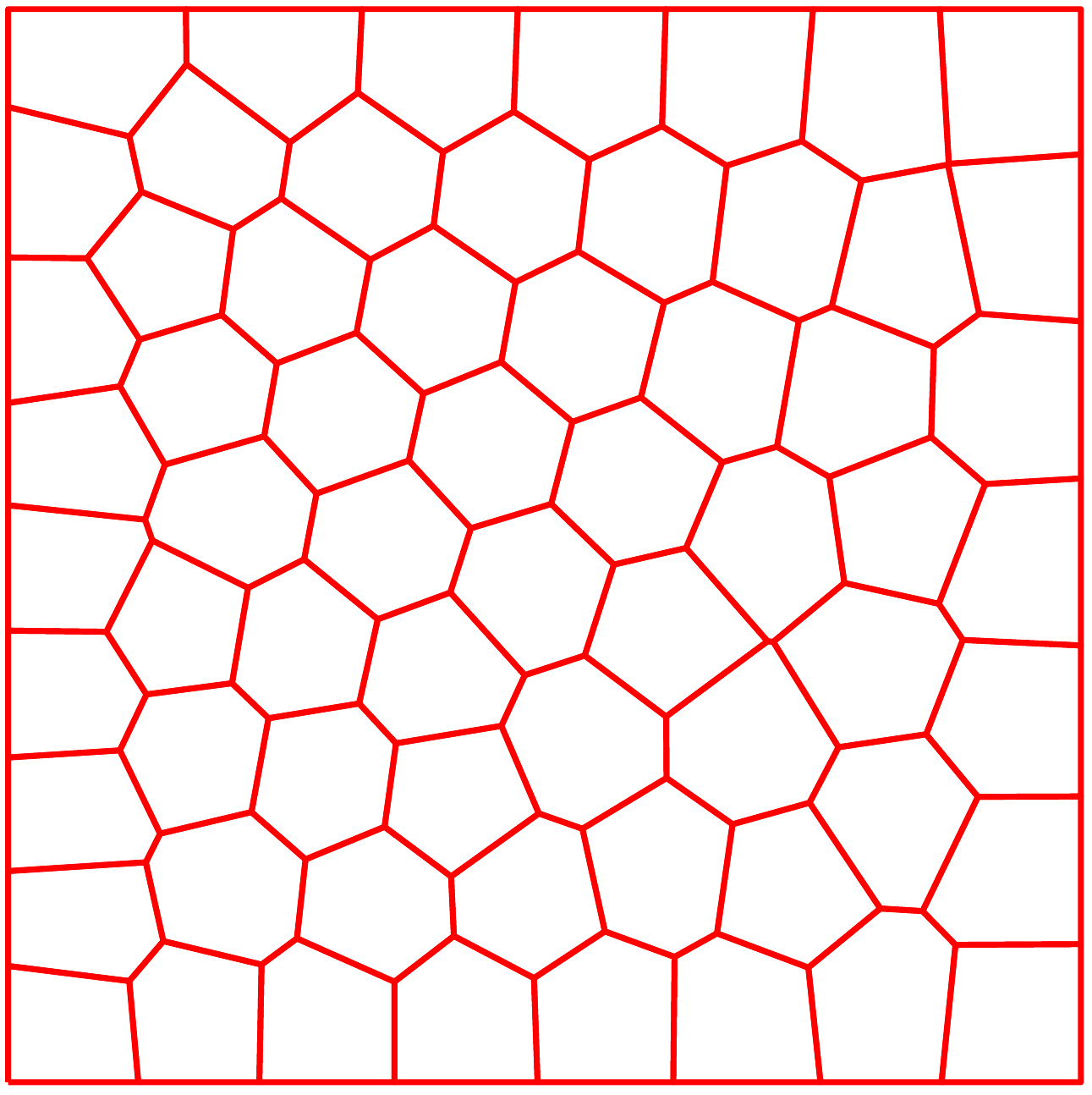}}} & \subfloat{
          \raisebox{-.45\height}{\includegraphics[width=0.17\textwidth]{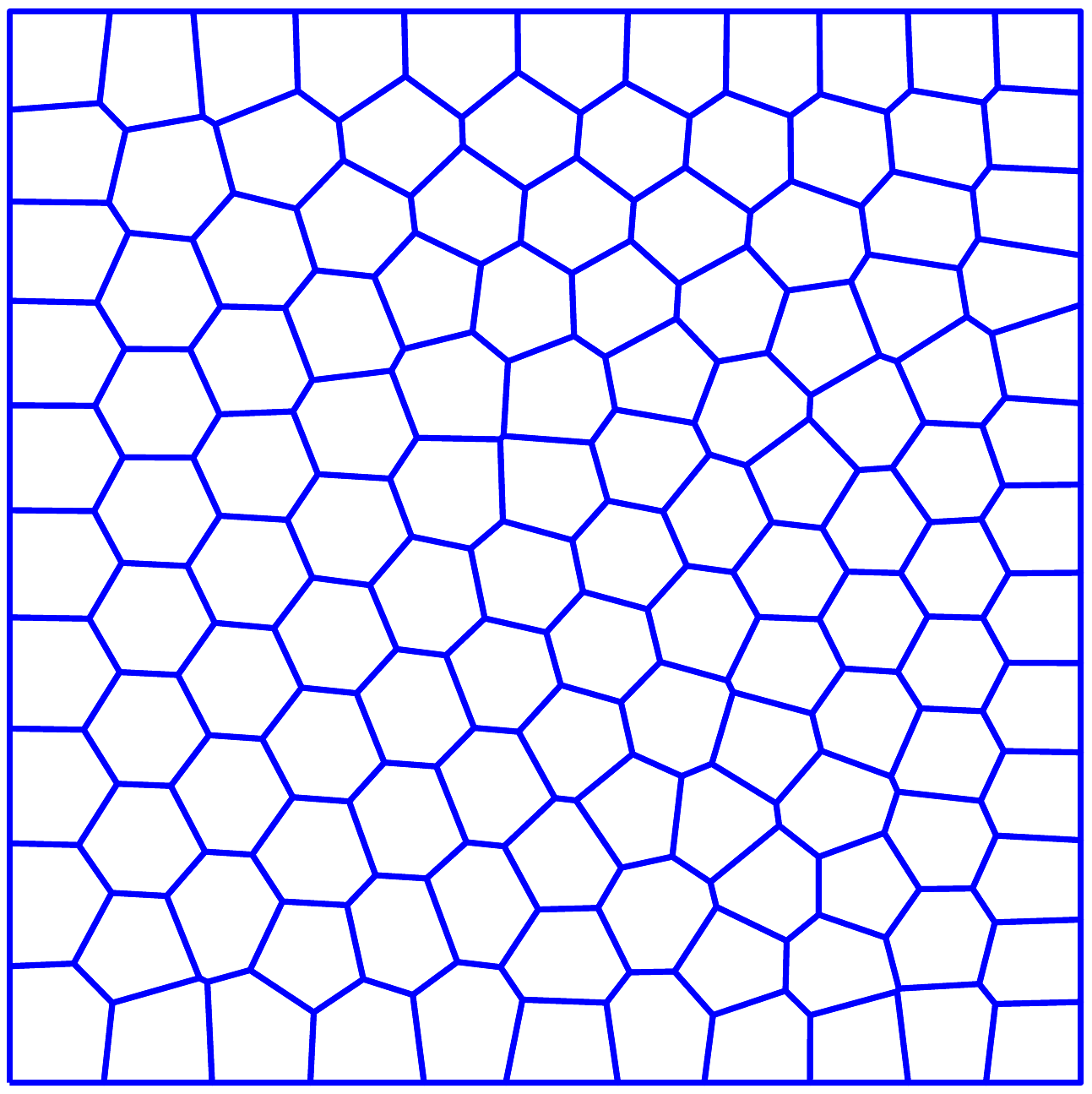}}} & \subfloat{
          \raisebox{-.45\height}{\includegraphics[width=0.17\textwidth]{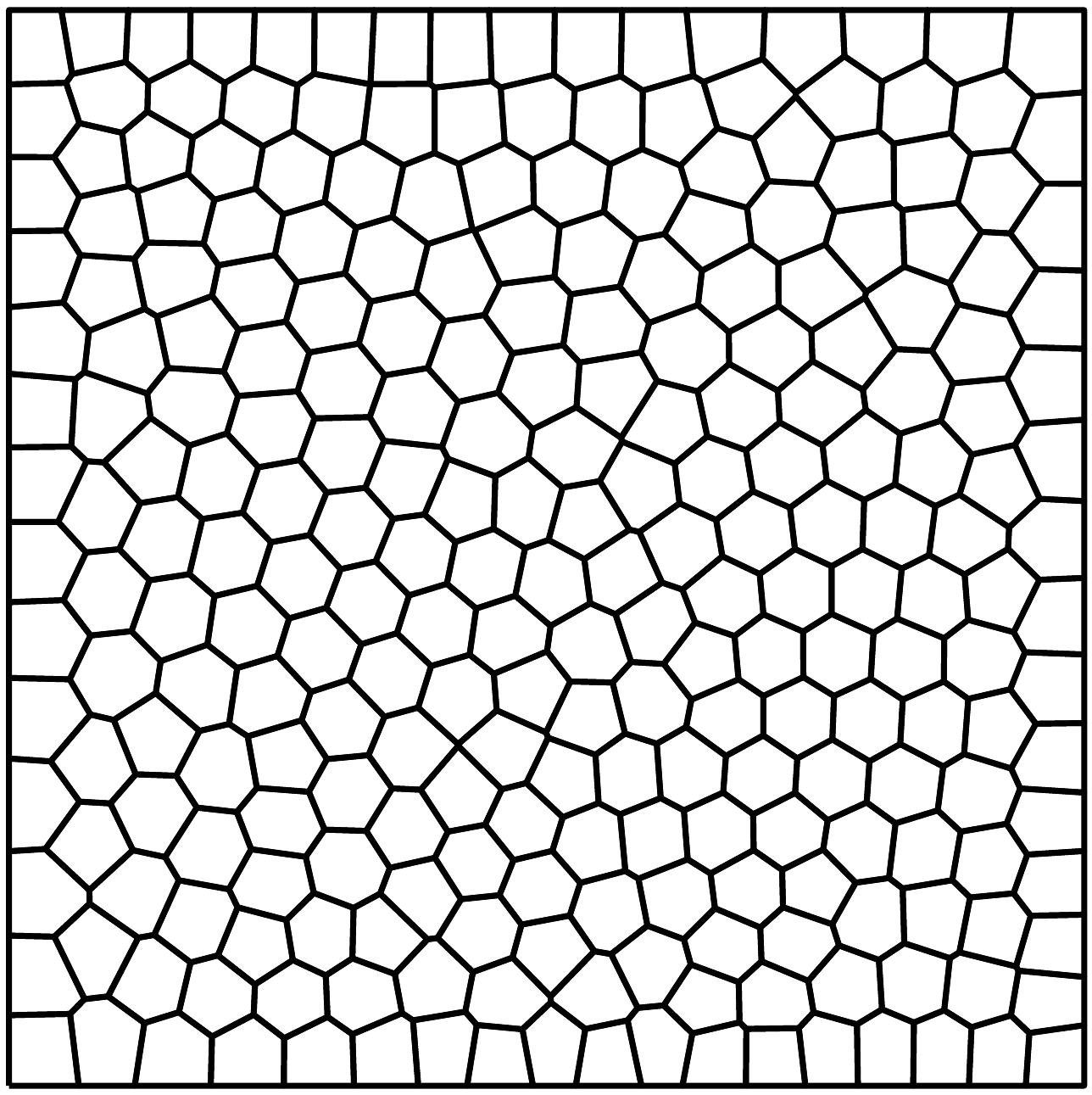}}}\\
Level 1 & \subfloat{
          \raisebox{-.45\height}{\includegraphics[width=0.17\textwidth]{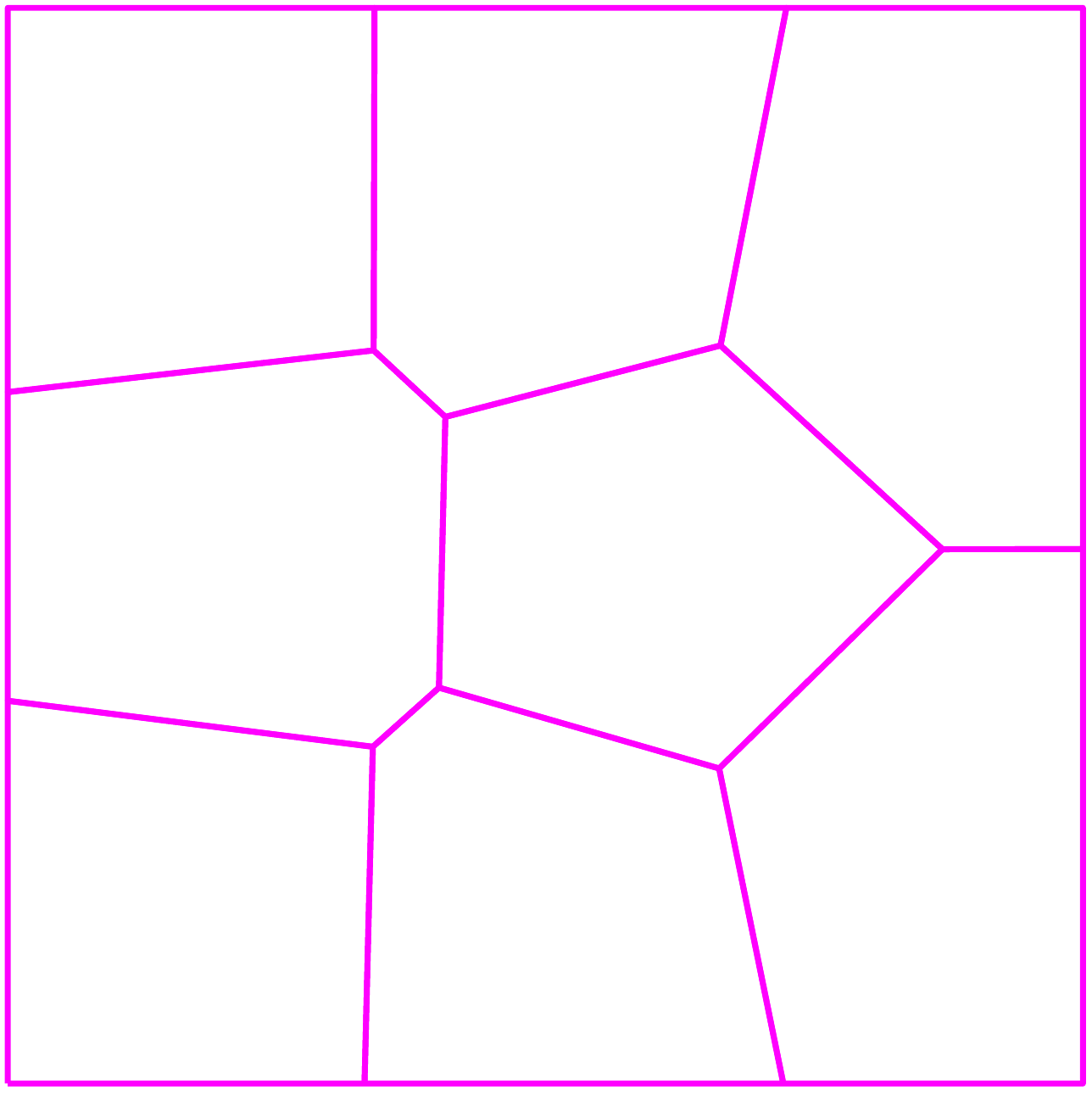}}}  & \subfloat{
          \raisebox{-.45\height}{\includegraphics[width=0.17\textwidth]{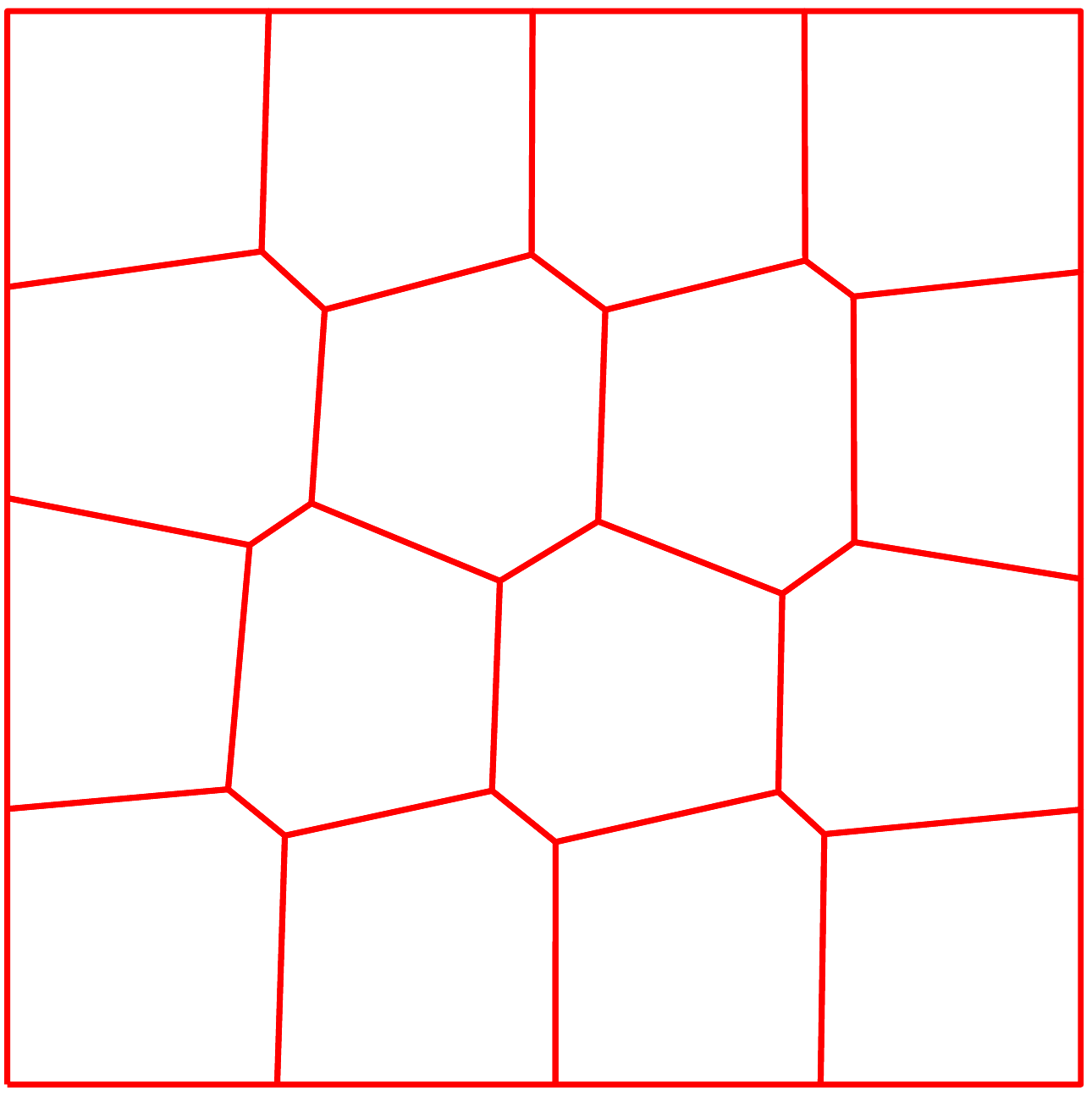}}} & \subfloat{
          \raisebox{-.45\height}{\includegraphics[width=0.17\textwidth]{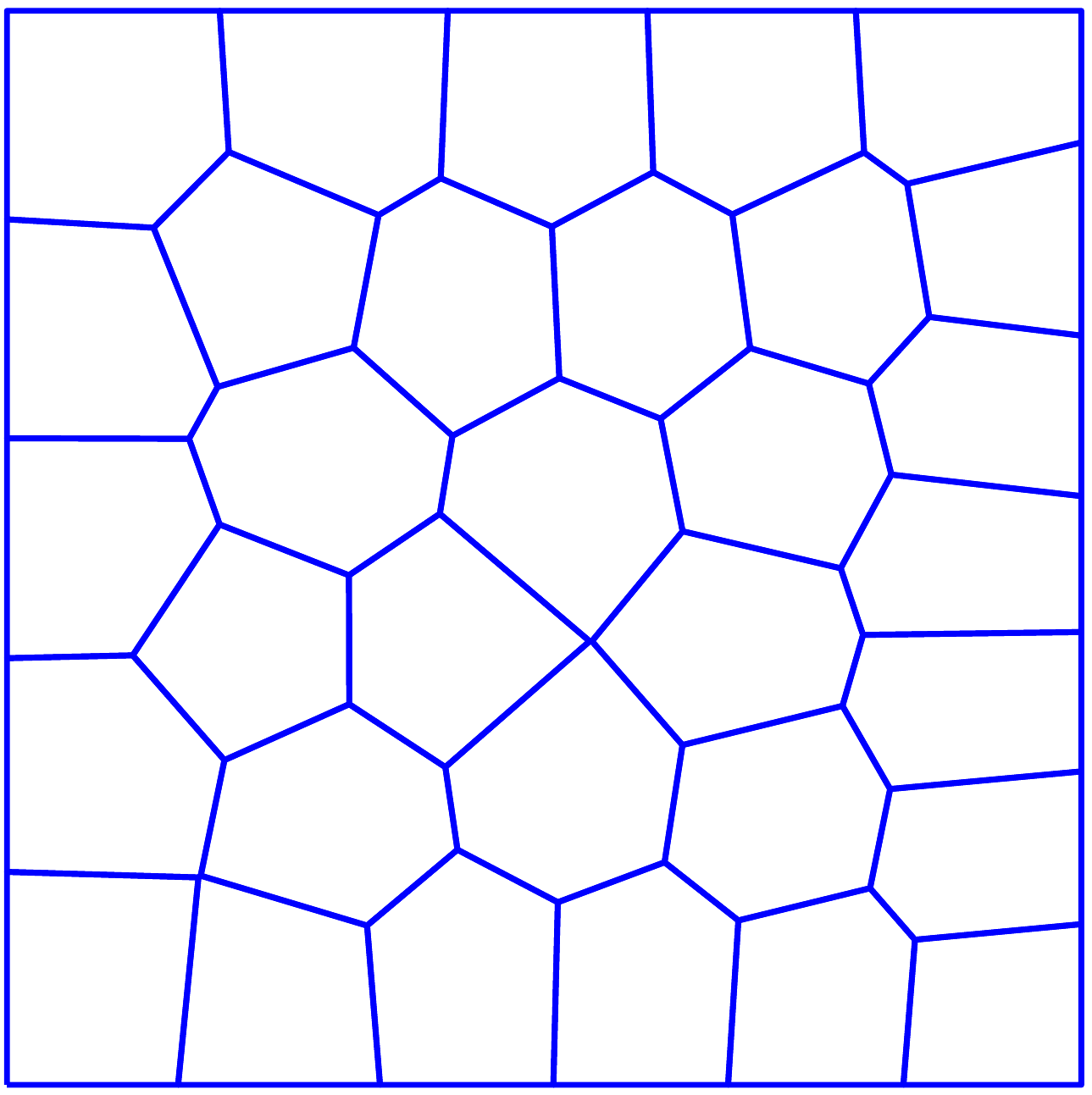}}} & \subfloat{
          \raisebox{-.45\height}{\includegraphics[width=0.17\textwidth]{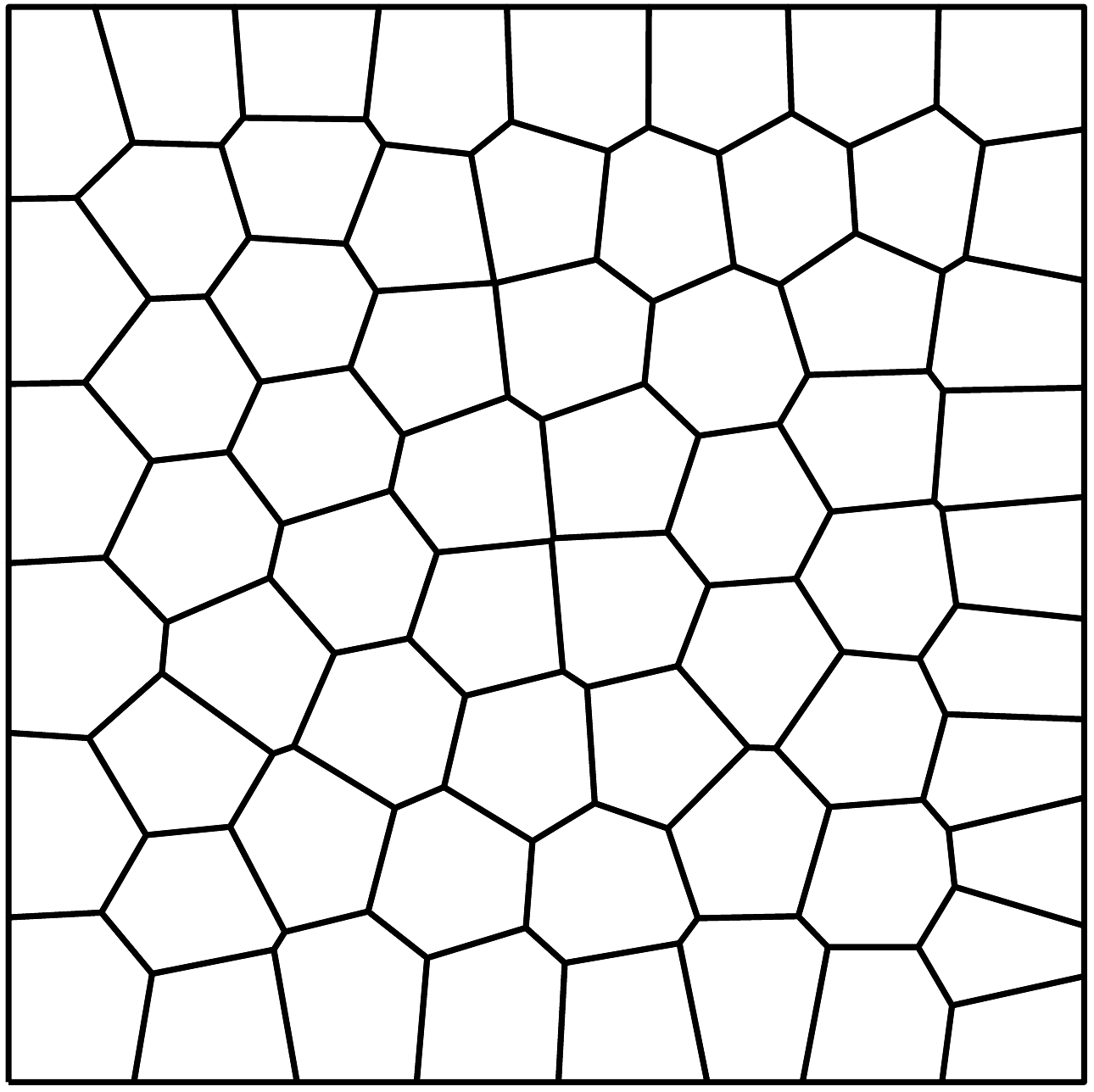}}}\\
\end{tabular}
\caption{Sets of non-nested grids employed for numerical simulations.}
\label{fig:grids}
\end{figure}

\begin{figure}[t]
\centering
\begin{tabular}{lcccc}
 &  &  &  & \\
 & \subfloat{
          \raisebox{-.45\height}{\includegraphics[width=0.17\textwidth]{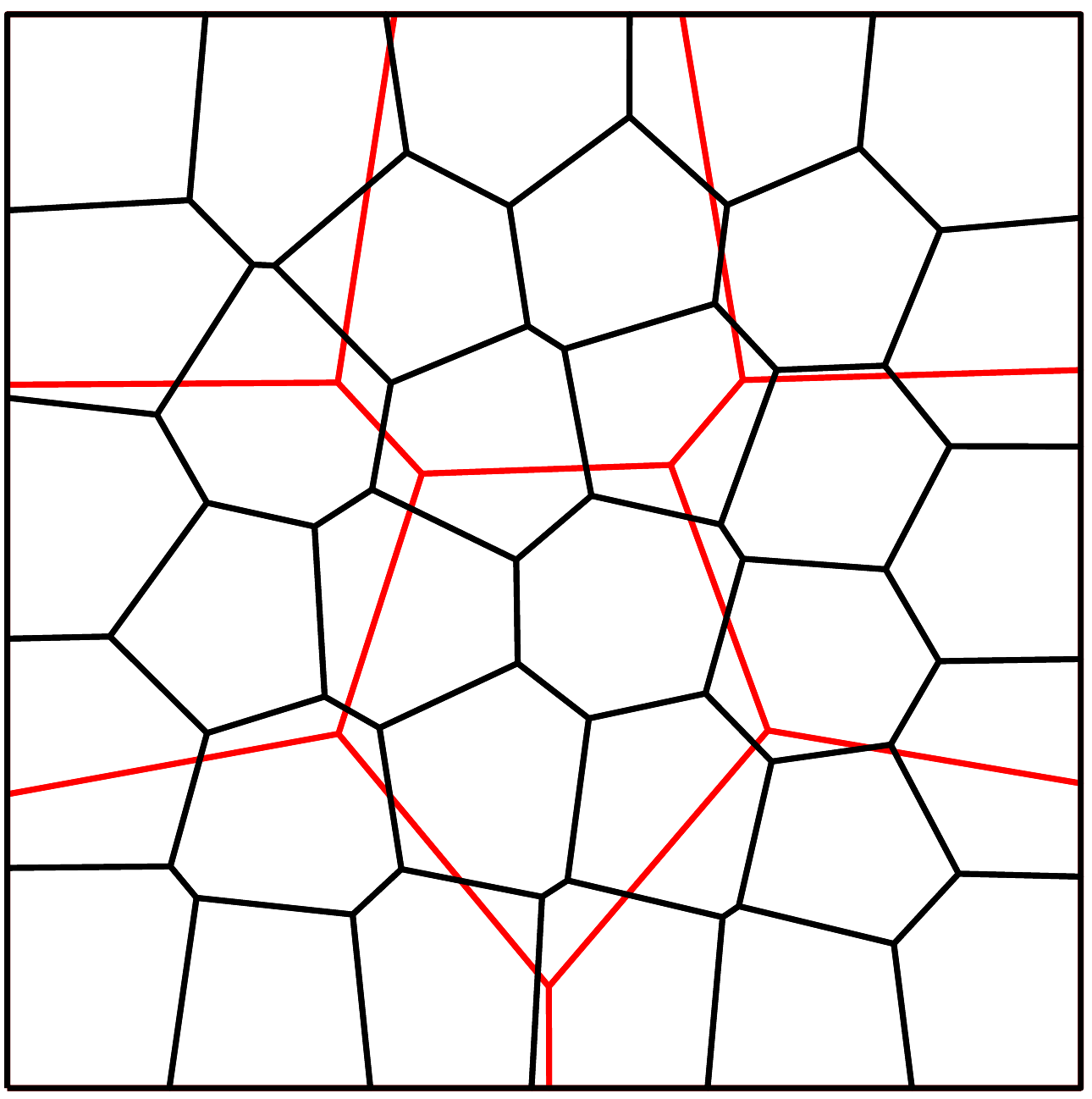}}}  & \subfloat{
          \raisebox{-.45\height}{\includegraphics[width=0.17\textwidth]{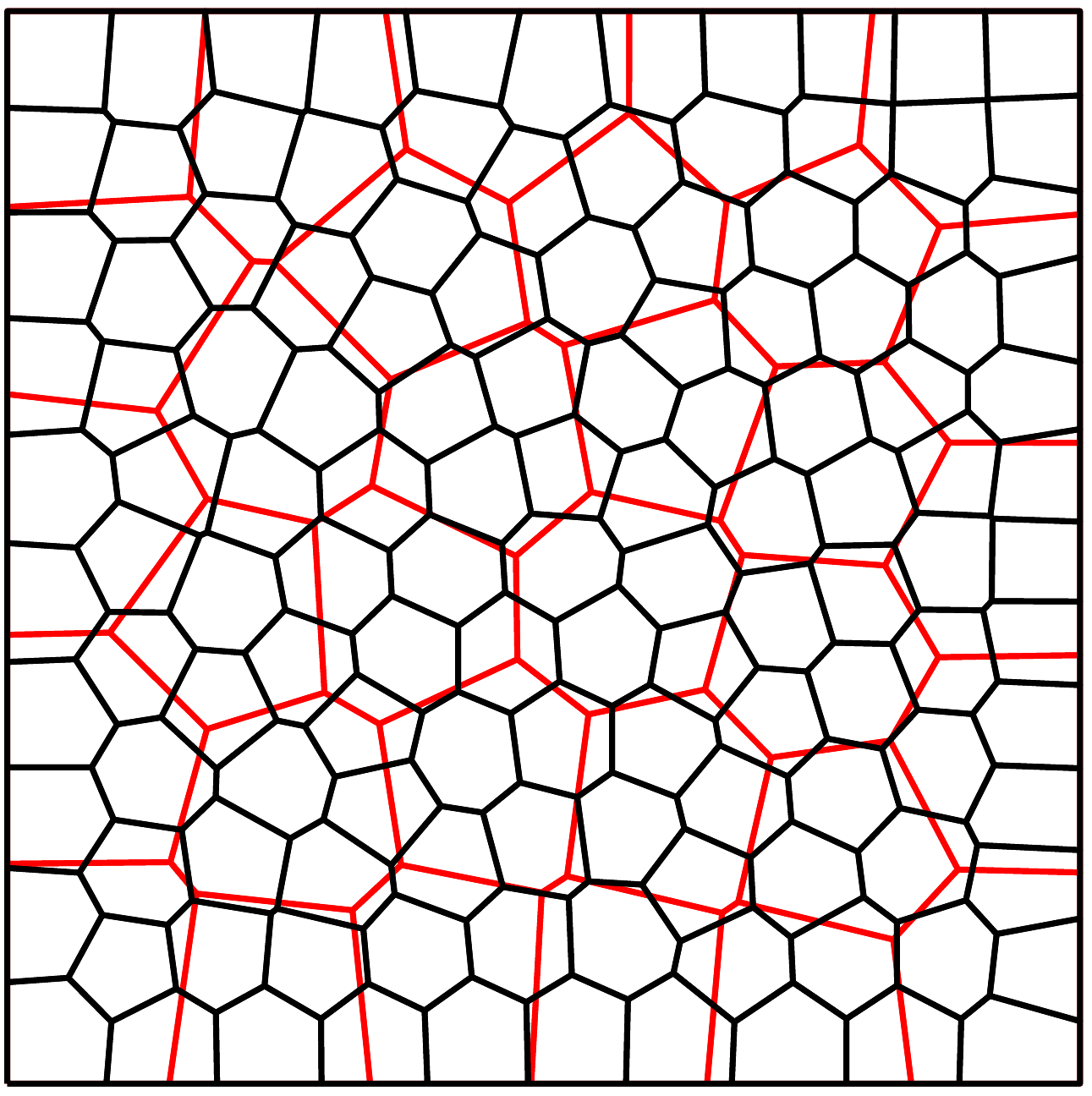}}} & \subfloat{
          \raisebox{-.45\height}{\includegraphics[width=0.17\textwidth]{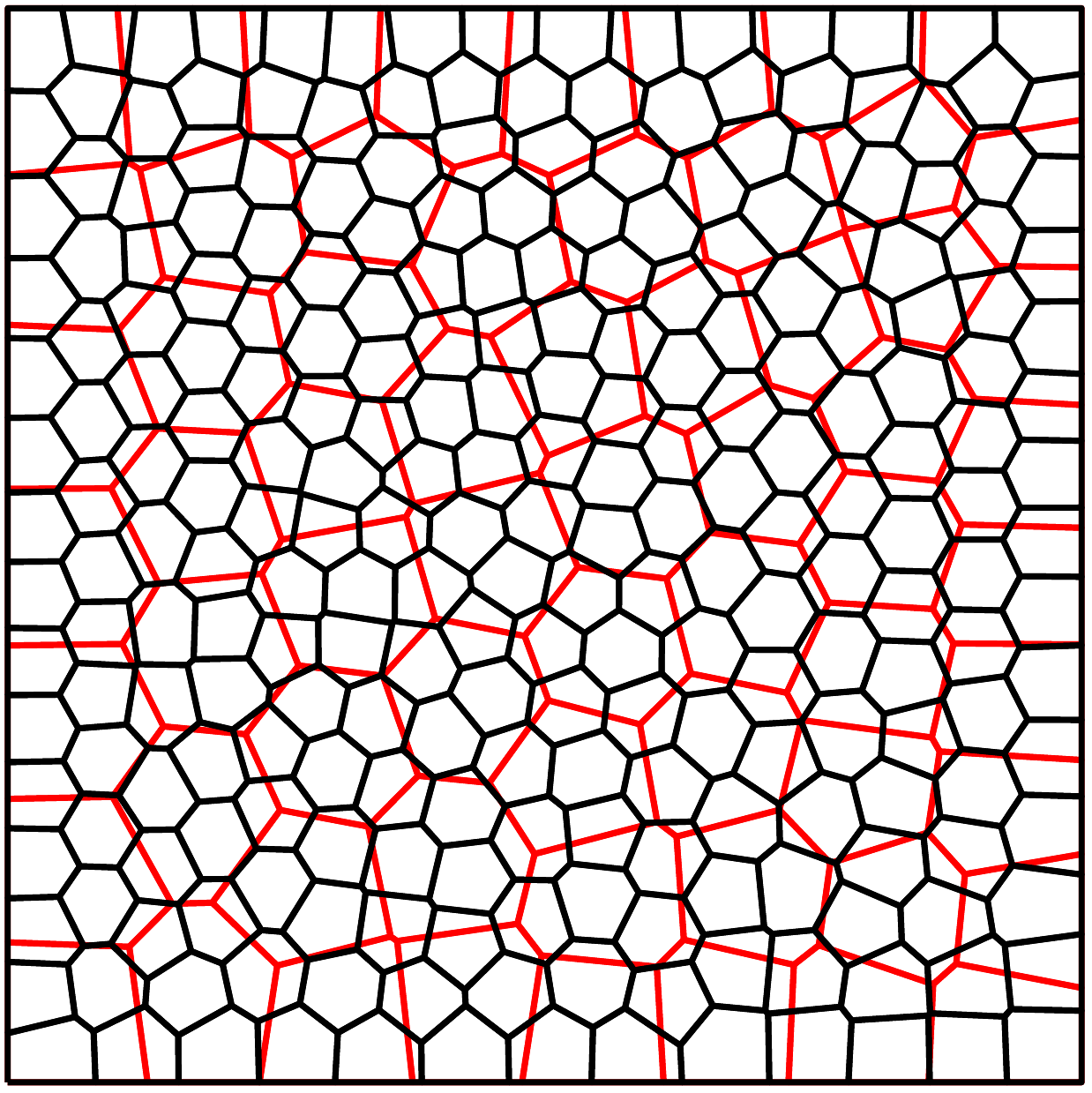}}} & \subfloat{
          \raisebox{-.45\height}{\includegraphics[width=0.17\textwidth]{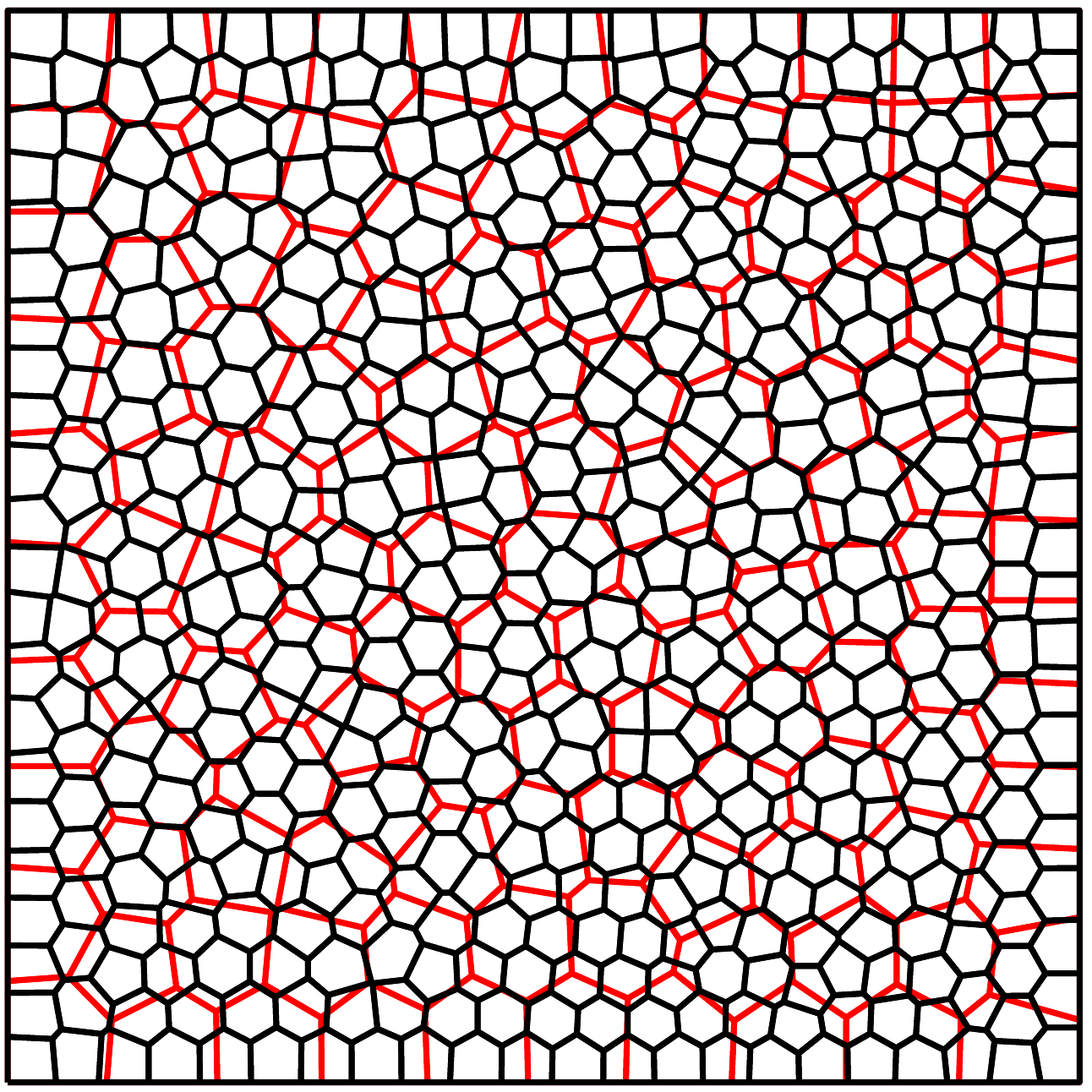}}}\\
\end{tabular}
\caption{Example of non-nested partition.}
\label{fig:non_nested}
\end{figure}

In this section we present several numerical results to test the theoretical convergence estimates provided in Theorem~\ref{thm:convergence}. We focus on a two dimensional problem on the unit square $\Omega = (0,1)^2$. For the simulations, we consider the sets of polygonal grids shown in Figure~\ref{fig:grids}. Each polygonal element mesh is generated through the Voronoi Diagram algorithm by using the software package \texttt{PolyMesher} \cite{Talietal12}. In particular the finest grids (Level 4) of Figure~\ref{fig:grids} consist of 512 (Set 1), 1024 (Set 2), 2048 (Set 3) and 4096 (Set 4) elements. Starting from the number of elements of each initial mesh, a sequence of  coarse, non-nested partitions is generated: each coarse mesh is built independently from the finer one, with the only constrain that the number of element is approximately $1/4$ of the finer one. An example of sequence of non-nested partitions is shown in Figure~\ref{fig:non_nested}. 
\medskip
\begin{figure}[t!]
\centering
\includegraphics[scale=0.35]{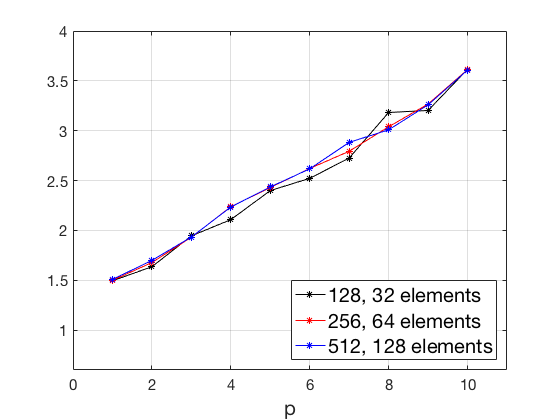}
\caption{Estimates of $\mathsf{C}_{\mathsf{stab}}(p)$ in Lemma~\ref{lemm:C_stab} as a function of $p$. Non-nested Voronoi meshes as shown in Figure~\ref{fig:grids}.}
\label{fig:estCstab}
\end{figure}

\begin{figure}[t!]
\centering
\includegraphics[scale=0.29]{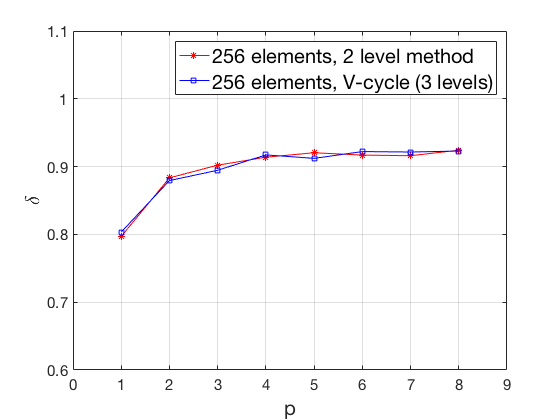}
\includegraphics[scale=0.29]{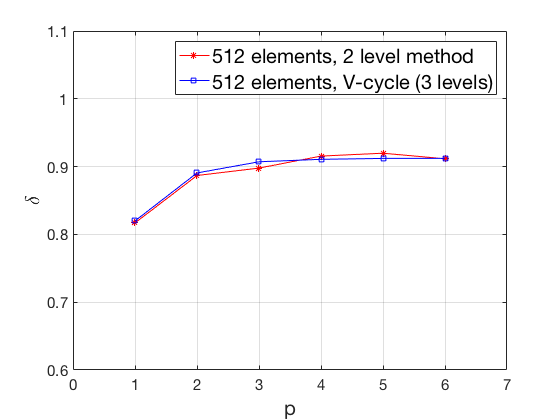}
\caption{Estimates of $\delta_2$ and $\delta_3$ in Theorem~\ref{thm:convergence} as a function of $p$, with $m_1=m_2=m=3p^2$ and two polyhedral grids of $256$ (left) and $512$ (right) elements.}
\label{fig:estDelta}
\end{figure}

First of all, we verify the estimate of Lemma~\ref{lemm:C_stab}, numerically evaluating $\mathsf{C}_{\mathsf{stab}}(p) \approx p$, where $p$ is the polynomial approximation degree. To this aim we consider three pairs of non-nested grids, where the number of elements of the coarser grid is the number of the finer divided by $4$: for each pair, we compute the value of $\mathsf{C}_{\mathsf{stab}}(p)$ as a function of $p$. Figure~\ref{fig:estCstab} show that, as expected, $\mathsf{C}_{\mathsf{stab}}(p)$ depends linearly on $p$ and is independent of the mesh-size $h$.
\vspace{6pt} %\medskip

We now consider the grids shown in Set 1 and in Set 2 of Figure~\ref{fig:grids}, and numerically evaluate the constant $\delta_j$ in Theorem~\ref{thm:convergence} based on selecting the Richardson smoother with $m_1=m_2=m=3p^2$, cf. Figure~\ref{fig:estDelta}. Here, we observe that 
$\delta_{2}$ and $\delta_3$ are asymptotically constant, as the polynomial
degree $p$ increases showing that our two-level and V-cycle algorithms are uniformly convergent also with respect to $p$ provided that $m \approx p^2$.
\vspace{6pt}

Next, we investigate the performance of the iterative Multigrid non-nested V-cycle algorithm presented in Sect.~\ref{sec:bpxvc}. In order to do that, we solve the Poisson problem with homogeneous Dirichlet boundary conditions on the unit square $\Om = (0,1)^2$, and we compute the number of iterations needed by our V-cycle algorithm to reduce the relative residual error below a given tolerance of $10^{-8}$, by varying the polynomial degree of approximation and the granularity of the finest grid. In Table~\ref{tab:hMGvsm_p1} we report the convergence factors 
\begin{equation}
\rho_J = \exp\left(\frac{1}{N_{it,J}}\ln \frac{\|\mathbf{r}_{N_{it,J}}\|}{\|\mathbf{r}_{0}\|}\right),
\end{equation}
where $N_{it,J}$ is the iteration counts needed to attain convergence of the $h$-version of the V-cycle scheme with $J$ levels, where $J=2,3,4$, while $\mathbf{r}_{N_{it,J}}$ and $\mathbf{r}_{0}$ are the final and initial residual vectors, respectively. Here the polynomial approximation degree on each level is chosen as $p_j =1$, $j=1,\dots,J$, while we vary the number of elements of the finest grid and the number of smoothing steps ($m_1=m_2=m$). According to Theorem~\ref{thm:convergence}, the convergence factor is independent from the spatial discretization step $h$, indeed, for a fixed $J\in\{2,3,4\}$ ad a fixed number of pre-smoothing steps $m$, the convergence factor is roughly constant between the 4 sets of grids. In particular, this means that the number of iterations needed from the proposed V-cycle method to attain the convergence is not influenced by the mesh refinement, contrarily of what we observe for the Conjugate Gradient (CG) method. As expected, the convergence factor is reduced by increasing the number of smoothing step.

%%%%%%%%%%%%%%
%%     TABLE 2 P=1     %%
%%%%%%%%%%%%%%
\begin{table}[t!]
\caption{Converge factors $\rho_J$ of V-cycle multigrid method as a function of $m$ and a comparison with the iteration counts of the CG method ($C_\sigma^j \equiv C_\sigma=10$, $p=1$).}

\centering
\footnotesize
\begin{tabular}{lccc||ccc}
\hline
 \rule{0pt}{2.5ex} \rule[-1.2ex]{0pt}{0pt}   & \multicolumn{3}{c||}{Set 1, $N_{iter}^{CG} = 410$} & \multicolumn{3}{c}{Set 2, $N_{iter}^{CG} = 600$}\\ 
 \hhline{~------}
 & 2 levels & 3 levels & 4 levels & 2 levels &  3 levels & 4 levels\\
\hline
$m=3$ & 0.77 & 0.83 & 0.83     & 0.82 & 0.84 & 0.85 \\
$m=5$ & 0.69 & 0.76 & 0.78     & 0.74 & 0.77 & 0.79 \\
$m=8$ & 0.63 & 0.69 & 0.72     & 0.66 & 0.70 & 0.73 \\
\hline
\hline
\rule{0pt}{2.5ex} \rule[-1.2ex]{0pt}{0pt}   & \multicolumn{3}{c||}{Set 3, $N_{iter}^{CG} = 867$} & \multicolumn{3}{c}{Set 4, $N_{iter}^{CG} = 1228$}\\ 
 \hhline{~------}
 & 2 levels & 3 levels & 4 levels &       2 levels  &  3 levels & 4 levels\\
\hline
$m=3$ & 0.79 & 0.85 & 0.93      & 0.78 & 0.84 & 0.87 \\
$m=5$ & 0.72 & 0.79 & 0.82      & 0.71 & 0.78 & 0.81 \\
$m=8$ & 0.65 & 0.72 & 0.76      & 0.64 & 0.72 & 0.74 \\
\hline
\end{tabular}
\label{tab:hMGvsm_p1}
\end{table}

%%%%%%%%%%%%%
%%    TABLE 2 P=3   %%
%%%%%%%%%%%%%
\begin{table}[t!]
\caption{Converge factors $\rho_J$ (and iterations count) of the V-cycle methods as a function of the number $m$ of pre-smoothing steps and comparison with the iteration counts of the CG method ($C_\sigma^j \equiv C_\sigma=10$, $p=3$).}

\centering
\footnotesize
\begin{tabular}{lccc||ccc}
\hline
\rule{0pt}{2.5ex} \rule[-1.2ex]{0pt}{0pt}   & \multicolumn{3}{c||}{Set 1, $N_{iter}^{CG} = 2212$} & \multicolumn{3}{c}{Set 2, $N_{iter}^{CG} = 3223$}\\ 
 \hhline{~------}
 & 2 levels & 3 levels & 4 levels & 2 levels &  3 levels & 4 levels\\
\hline
$m=3$ &  0.99 (3306) & 0.98 (992) & 0.98 (955)     & 0.97 (616) & 0.98 (852) & 0.98 (1024) \\
$m=5$ &  0.96 (429)   & 0.97 (566) & 0.97 (591)     & 0.95 (396) & 0.96 (523) & 0.97 (626) \\
$m=8$ &  0.94 (296)   & 0.95 (367) & 0.95 (388)     & 0.94 (277) & 0.95 (339) & 0.95 (403) \\
\hline
\hline
\rule{0pt}{2.5ex} \rule[-1.2ex]{0pt}{0pt}   & \multicolumn{3}{c||}{Set 3, $N_{iter}^{CG} = 4174$} & \multicolumn{3}{c}{Set 4, $N_{iter}^{CG} = 6689$}\\ 
 \hhline{~------}
 & 2 levels  & 3 levels & 4 levels &   2 levels &  3 levels & 4 levels\\
\hline
$m=3$ & -                & 0.98 (1061) & 0.98 (860)      & -                & 0.97 (699) &  0.98 (823) \\
$m=5$ & 0.96 (428) & 0.97 (648)   & 0.97 (527)      & 0.95 (392) & 0.96 (435) &  0.96 (508) \\
$m=8$ & 0.94 (288) & 0.96 (418)   & 0.95 (341)      & 0.93 (273) & 0.94 (290) &  0.95 (335) \\
\hline
\hline
\end{tabular}
\label{tab:hMGvsm_p3}
\end{table}

We have repeated the same set of experiments employing $p_j=3,\ \forall j=1,\dots,J$; the results are reported in Table~\ref{tab:hMGvsm_p3}, where we also have reported the iterations count (between parenthesis). Firstly, a comparison between Table~\ref{tab:hMGvsm_p1} and Table~\ref{tab:hMGvsm_p3} confirms that the convergence factor increases as $p$ grows up if the number of smoothing steps is kept fixed. Secondly, we observe that if the number of smoothing step is kept too small then the convergence of the method could not be guaranteed: indeed, according to Theorem~\ref{thm:convergence}, a uniformly convergent (also with respect to $p$) solver require a number of smoothing steps $m > 2C_1 C_Q \gtrsim p^{2+\mu}$ as shown in Figure~\ref{fig:estDelta}. If $m$ is big enough, we observe that also in this case the number of iterations does not depend from the granularity of the underlying mesh, while the iterations count of the Conjugate Gradient method is growing if $h$ decrease.

%%%%%%%%%%%%%%%%%%%%%%%%%%%%%%%%%%%
%%%%%          ADDITIVE SCHWARZ SMOOTHER          %%%%%%
%%%%%%%%%%%%%%%%%%%%%%%%%%%%%%%%%%%
\section{Additive Schwarz smoother}\label{sec:ASSmoth}
In order to improve the convergence properties of the V-cycle algorithm studied above, we define in this section a domain decomposition preconditioner that we will use as a smoothing operator instead of the Richardson iteration. To this end, let $\mcal[T][j]$ and $\mcal[T][j-1]$ be  respectively the finer and the coarser non-nested meshes, satisfying the grid assumptions given in Sect.~\ref{subsection_Grid_assumption}. We then introduce the \textit{local} and \textit{coarse} solvers, that are the key ingredients in the definition of the smoother on the space $V_j,\ j=2,\dots,J$.
\medskip

\textbf{Local Solvers}. Let us consider the finest mesh $\mcal[T][j]$ with cardinality $N_j$, then for each element $\elem_i \in\mcal[T][j]$, we define a local space $V_j^i$ as the restriction of the DG finite element space $V_j$ to the element $\elem_i \in \mcal[T][j]$:
\begin{equation}
V_j^i = V_j|_{\elem_i} \equiv \mcal[P][p_j](\elem_i) \qquad \forall i = 1,...,N_j,
\end{equation}
and for each local space, the associated local bilinear form is defined by
\begin{equation}
\mathcal{A}_j^i: V_j^i \times V_j^i \rightarrow \mathbb{R}, \quad \mathcal{A}_j^i(u_i,v_i) = \Aa[j][R_i^T u_i][R_i^T v_i] \quad \forall u_i,v_i \in V^i,
\end{equation}
where $R_i^T:V_j^i \rightarrow V_j$ denotes the classical extension by-zero operator from the local space $V_j^i$ to the global $V_j$.
\medskip

\textbf{Coarse Solver.} The natural choice in our contest is to define the coarse space $V^0_j$ to be exactly the same used for the \textit{Coarse grid correction} step of the V-cycle algorithm introduced in Sect.~\ref{sec:bpxvc}, that is 
\begin{equation} 
V^0_j = V_{j-1} \equiv \{v\in L^2(\Om):v|_\elem\in \mcal[P][p_{j-1}](\elem),\elem\in\mcal[T][j-1]\},
\end{equation}
the bilinear form on $V^0_j$ is then given by 
\begin{equation} 
\mathcal{A}_j^0: V_j^0 \times V_j^0 \rightarrow \mathbb{R}, \quad \mathcal{A}_j^0(u_0,v_0) = \Aa[j-1][u_0][v_0] \quad \forall u_0,v_0 \in V_j^0.
\end{equation}
We also define the injection operator from $V_j^0$ to $V_j$: conversely with respect to the case where the coarser mesh is obtained by agglomeration, here the injection operator is not trivial, and it is defined as the prolongation operator introduced in Sect.~\ref{sec:bpxvc}, that is $R_0^T:V_j^0 \rightarrow V_j,\ R_0^T = I_{j-1}^j.$ By introducing the projection operators $P_i = R_i^T \tilde{P_i}: V_j \rightarrow V_j, \text{ }i=0,1,\dots,N_j,$ where 
\begin{align}
&\tilde{P_i}:V_j \rightarrow V^i_j, \quad \mathcal{A}_j^i(\tilde{P_i}v_h, w_i) = \Aa[j][v_h][R_i^T w_i] \quad \forall w_i \in V^i_j, \quad i=1,\dots,N_j,\\
&\tilde{P_0}:V_j \rightarrow V^0_j, \quad \mathcal{A}_j^0(\tilde{P_0}v_h, w_0) = \Aa[j][v_h][R_0^T w_0] \quad \forall w_0 \in V^0_j,
\end{align}
the additive Schwarz operator is defined by $P_{ad} = \sum_{i=0}^{N_j} (R_i^T (A^i_j)^{-1} R_i) A_j \equiv B^{-1}_{ad}A_j,$ where $B^{-1}_{ad} = \sum_{i=0}^{N_j} (R_i^T (A^i_j)^{-1} R_i) $ is the preconditioner. Then, the \textit{Additive Schwarz} smoothing operator with $m$ steps consists in performing $m$ iterations of the \textit{Preconditioned Conjugate Gradient} method using $B_{ad}$ as preconditioner. In Algorithm~\ref{alg:ASMG} we outline the V-cycle multigrid method using $P_{ad}$ as a smoother. Here, $\mathsf{MG_{\mathcal{AS}}} (j, g, z_0,m_1,m_2)$ denotes the approximate solution of $A_jz=g$ obtained after one iteration, with initial guess $z_0$ and $m_1$, $m_2$ pre- and post-smoothing steps, respectively. Here, the smoothing step is performed by the algorithm $ASPCG$, i.e., $z = ASPCG(A, z_0, g, m)$ represents the output of $m$ steps of \textit{Preconditioned Conjugate Gradient} method applied to the linear system of equations $Ax = g,$ by using $B_{as}$ as preconditioner and starting with the initial guess $z_0$.

\begin{algorithm}[t!]
\caption{One iteration of Multigrid V-cycle scheme with AS-smoother}
\label{alg:ASMG}
\begin{algorithmic}
\State \underline{{\it  Pre-smoothing}}:
\If{j=1}
\State $\mathsf{MG}_\mathcal{AS} (1,g,z_0,m_1,m_2) = A_1^{-1}g.$
\Else
\State \underline{{\it  Pre-smoothing}}:
\State $z^{(m_1)} = ASPCG(A_j, z_0, g, m_1)$;\vspace{0.3cm}
\State \underline{{\it  Coarse grid correction}}:
\State $r_{j-1} = I_j^{j-1}(g-A_jz^{(m_1)})$;
\State $e_{j-1} = \mathsf{MG}_\mathcal{AS} (j-1,r_{j-1},0,m_1,m_2)$;
\State $z^{(m_1+1)}=z^{(m_1)}+I_{j-1}^je_{j-1}$;\vspace{0.3cm}
\State \underline{{\it  Post-smoothing}}:
\State $z^{(m_1+m_1+1)} = ASPCG(A_j, z^{(m_1+1)}, g, m_2);$\vspace{0.3cm}
\State $\mathsf{MG}_\mathcal{AS} (j,g,z_0,m_1,m_2)=z^{(m_1+m_2+1)}.$
\EndIf
\end{algorithmic}
\end{algorithm}

The numerical performance of Algorithm~\ref{alg:ASMG} are reported in Tables~\ref{tab:hMGvsm_p1_AS}, \ref{tab:hMGvsm_p3_AS} and \ref{tab:hMGvsm_p5_AS}, for the corresponding V-cycle algorithm with $J=2,3,4$ levels. The simulations are similar to the ones described in the previous section: here we used the grids of Set 2, 3 and 4 of Figure~\ref{fig:grids}, and we varied the polynomial degree $p \in \{1,3,5\}$. Firstly, we observe that, also in this case, the number of iteration does not increase with the number of elements in the underlying mesh for a fixed number of smoothing steps $m$; moreover, we does not observe the constrain from below required to the number of smoothing steps with respect to the degree of approximation: the method converges also with high degree of approximation and a small number of smoothing steps.
Finally, Table~\ref{tab:hMGvsm_p43_AS} shows the numerical results relatives to an example of $hp$-multigrid, characterized by a choice of different polynomial degrees of approximation between non-nested space: also in this case we observe that the number of iterations is independent of the granularity of the finest mesh, and we have convergence for any choice of smoothing steps $m$.  

%%%%P1
\begin{table}[t!]
\caption{Iteration counts of the V-cycle solvers with the Additive Schwarz smoother as a function of $m$ ($C_\sigma^j \equiv C_\sigma=10$, $p=1$).}
\centering
\footnotesize
\begin{tabular}{|lccc|ccc|ccc|}
\hline
 & \multicolumn{3}{c|}{Set 2}  & \multicolumn{3}{c|}{Set 3} & \multicolumn{3}{c|}{Set 4} \\ 
\hhline{~---------}
 & 2 lev. &  3 lev. & 4 lev.  & 2 lev. &  3 lev. & 4 lev.   & 2 lev. &  3 lev. & 4 lev.  \\
\hline
$m=3$    & 18   & 18  & 18        & 18   & 18   & 18       & 20   & 20  & 20  \\
$m=5$    & 9     & 9    & 9          & 9     & 9     & 9         & 10   & 10  & 10  \\
$m=8$    & 5     & 5    & 5          & 5     & 5     & 5         & 5     & 5    & 5    \\
\hline
\rule{0pt}{2.5ex} \rule[-1.2ex]{0pt}{0pt} & \multicolumn{3}{c|}{$N_{iter}^{CG} = 600$} & \multicolumn{3}{c|}{$N_{iter}^{CG} = 867$} & \multicolumn{3}{c|}{$N_{iter}^{CG} = 1228$}\\  
\hline
\end{tabular}
\label{tab:hMGvsm_p1_AS}
\end{table}

%%%%%P3
\begin{table}[t!]
\caption{Iteration counts of the V-cycle solvers with the Additive Schwarz smoother as a function of $m$ ($C_\sigma^j \equiv C_\sigma=10$, $p=3$).}
\centering
\footnotesize
\begin{tabular}{|lccc|ccc|ccc|}
\hline
 & \multicolumn{3}{c|}{Set 2}  & \multicolumn{3}{c|}{Set 3} & \multicolumn{3}{c|}{Set 4} \\ 
 \hhline{~---------}
 & 2 lev. &  3 lev. & 4 lev.  & 2 lev. &  3 lev. & 4 lev.   & 2 lev. &  3 lev. & 4 lev.  \\
\hline
$m=3$          & 63   & 64  & 64     & 57   & 59   & 59       & 59   & 60  & 60  \\
$m=5$          & 27   & 27  & 27     & 25   & 25   & 25       & 26   & 26  & 26  \\
$m=8$          & 13   & 13  & 13     & 13   & 14   & 14       & 14   & 14  & 14  \\
\hline
\rule{0pt}{2.5ex} \rule[-1.2ex]{0pt}{0pt}  & \multicolumn{3}{c|}{$N_{iter}^{CG} = 3223$} & \multicolumn{3}{c|}{$N_{iter}^{CG} = 4174$} & \multicolumn{3}{c|}{$N_{iter}^{CG} = 6689$}\\  
\hline
\end{tabular}
\label{tab:hMGvsm_p3_AS}
\end{table}

%%%%P5
\begin{table}[t!]
\caption{Iteration counts of the V-cycle solvers with the Additive Schwarz smoother as a function of $m$ ($C_\sigma^j \equiv C_\sigma=10$, $p=5$).}
\centering
\footnotesize
\begin{tabular}{|lccc|ccc|ccc|}
\hline
 & \multicolumn{3}{c|}{Set 2}  & \multicolumn{3}{c|}{Set 3} & \multicolumn{3}{c|}{Set 4} \\ 
 \hhline{~---------}
 & 2 lev. &  3 lev. & 4 lev.  & 2 lev. &  3 lev. & 4 lev.   & 2 lev. &  3 lev. & 4 lev.  \\
\hline
$m=3$   & 148 & 156 & 156     & 125   & 132 & 132     & 149  & 158 & 157  \\
$m=5$   & 59   & 59   & 58       & 51     & 51   & 51       & 59    & 60   & 60  \\
$m=8$   & 26   & 26   & 26       & 24     & 24   & 24       & 27    & 27   & 27  \\
\hline
\rule{0pt}{2.5ex} \rule[-1.2ex]{0pt}{0pt}  & \multicolumn{3}{c|}{$N_{iter}^{CG} = 7676$} & \multicolumn{3}{c|}{$N_{iter}^{CG} = 11525$} & \multicolumn{3}{c|}{$N_{iter}^{CG} = 15814$}\\  
\hline
\end{tabular}
\label{tab:hMGvsm_p5_AS}
\end{table}

%%%%%HP-MULTIGRID
\begin{table}[t!]
\caption{Iteration counts of the $hp$-version of the V-cycle solvers with the Additive Schwarz smoother as a function of $m$. Here the polynomial degree on each space is $p_j = j$ for $j = 1,2,3,4$.}
\centering
\footnotesize
\begin{tabular}{|lccc|ccc|ccc|}
\hline
 & \multicolumn{3}{c|}{Set 2}  & \multicolumn{3}{c|}{Set 3} & \multicolumn{3}{c|}{Set 4} \\ 
 \hhline{~---------}
 & 2 lev. &  3 lev. & 4 lev.  & 2 lev. &  3 lev. & 4 lev.   & 2 lev. &  3 lev. & 4 lev.  \\
\hline
$m=3$   & 85   & 86   & 86       & 79     & 80   & 80       & 83     & 85   & 84  \\
$m=5$   & 35   & 35   & 35       & 32     & 32   & 32       & 33     & 33   & 33  \\
$m=8$   & 17   & 17   & 17       & 17     & 17   & 17       & 17     & 18   & 17  \\
\hline
\rule{0pt}{2.5ex} \rule[-1.2ex]{0pt}{0pt}  & \multicolumn{3}{c|}{$N_{iter}^{CG} = 5108$} & \multicolumn{3}{c|}{$N_{iter}^{CG} = 7697$} & \multicolumn{3}{c|}{$N_{iter}^{CG} = 10572$}\\  
\hline
\end{tabular}
\label{tab:hMGvsm_p43_AS}
\end{table}

%%%%%%%%%%%%%%%%%%%%%%%%
%%%%%%%%    APPENDIX    %%%%%%%%
%%%%%%%%%%%%%%%%%%%%%%%%
\appendix

%%%%%%%%%%%%%%%%%%%%%%%%
%%%%%% TRACE INEQUALITY %%%%%%%
%%%%%%%%%%%%%%%%%%%%%%%%
\section{Proof of Lemma~\ref{lem:inversecont}}\label{appx:trace}
\begin{proof}[of Lemma~\ref{lem:inversecont}]
We follow the idea of \cite[Proof of Lemma~1.49]{DiPiEr}.
First of all, we observe that 
\begin{equation}\label{eq:L2boundary}
\normL[v][2][\partial \elem][2] = \sum_{F \subset \partial \elem} \normL[v][2][F][2].
\end{equation} 
For each face $F \subset \partial \elem$ let $T_F \subset \elem$ be a $d$-dimensional simplex sharing the face $F$ with $\elem$ and satisfying the Assumption~\ref{ass1}: in $T_F$ we define a function $\sigma_F$ as follow: 
\begin{equation}
\sigma_F: \textbf{x} \in \overline{T_F} \mapsto \sigma_F(\textbf{x}) = \frac{| F |}{d |T_F|} (\textbf{x} - \textbf{v}_F),
\end{equation} 
where $\textbf{v}_F$ is the vertex of the simplex $T_F$ opposite to the face $F$. We observe that:
\begin{itemize}
\item $\sigma_F(\textbf{x}) \cdot \textbf{n}_F = \frac{| F |}{d |T_F|} \tilde{h}\ \forall\textbf{x} \in F$, where $\tilde{h}$ is the height of the simplex respect to the face $F$, that is also $\tilde{h} = \frac{d |T_F|}{| F|}$, then $\sigma_F|_F \cdot \textbf{n}_F = 1$;
\item $\sigma_F|_{F'} \cdot \textbf{n}_{F'} = 0\ \forall \text{ faces } F' \subset \partial T_F, F' \ne F$; 
\end{itemize}
then we have:
\begin{align}
\normL[v][2][F][2]  & = \int_{F} | v |^2 d\sigma = \int_{\partial T_F} | v |^2 \sigma_F \cdot \textbf{n}_F d\sigma = \int_{T} \nabla \cdot (|v|^2 \sigma_F) d\textbf{x} \\ &= \int_{T} 2v\nabla v \cdot \sigma_F d\textbf{x} + \int_{T} |v|^2 \nabla \cdot  \sigma_F d\textbf{x};
\end{align}
now the following properties hold for $\sigma_F$:
\begin{itemize}
\item $\nabla \cdot \sigma_F = \nabla \cdot \frac{| F |}{d |T_F|} (\textbf{x} - \textbf{v}_F) = \frac{| F |}{d |T_F|} \nabla \cdot \textbf{x} = \frac{| F |}{|T_F|}$;
\item $\| \sigma_F \|_{[L^{\infty}(T_F)]^d} = \frac{| F |}{d |T_F|} h_{T} \le \frac{| F |}{d |T_F|} h_{\elem}$, 
\end{itemize}
\noindent which implies:
\begin{align}
\normL[v][2][F][2] & \le 2 \| \sigma_F \|_{[L^{\infty}(T_F)]^d} \| v \nabla v \|_{[L^1(T_F)]^d} + \frac{| F |}{|T_F|} \| v \|_{L^2(T_F)}^2, \\
& \le 2 \frac{| F |}{d |T_F|} h_{\elem} \| v \|_{L^2(T_F)} | v |_{H^1(T_F)} + \frac{| F |}{|T_F|} \| v \|_{L^2(T_F)}^2, 
\end{align}
using the Assumption~\ref{ass1} we have
\begin{equation}
\normL[v][2][F][2] \le 2C \| v \|_{L^2(T_F)} | v |_{H^1(T_F)} + \frac{Cd}{h_{\elem}} \| v \|_{L^2(T_F)}^2.
\end{equation}
By using Young Inequality we could bound 
\begin{equation}
\| v \|_{L^2(T_F)} | v |_{H^1(T_F)} \le \frac{1}{2} \Bigl( \frac{\epsilon}{h_{\elem}} \| v \|_{L^2(T_F)}^2 + \frac{h_{\elem}}{\epsilon} | v |_{H^1(T_F)}^2 \Bigr),
\end{equation}
where we have chosen $\epsilon \ge 1$. Using the previous inequality we have
\begin{equation}\label{eq:trace_face}
\normL[v][2][F][2] \le 2Cd\Bigl( \frac{\epsilon}{h_{\elem}} \| v \|_{L^2(T_F)}^2 + \frac{h_{\elem}}{\epsilon} | v |_{H^1(T_F)}^2 \Bigr).
\end{equation}
we observe that \eqref{eq:trace_face} holds $\forall F \subset \partial \elem$. Then, thanks to \eqref{eq:L2boundary}, we have
\begin{align}
\normL[v][2][\partial \elem][2] & = \sum_{F \subset \partial \elem} \normL[v][2][F][2] \le \sum_{F \subset \partial \elem} 2Cd\Bigl( \frac{\epsilon}{h_{\elem}} \| v \|_{L^2(T_F)}^2 + \frac{h_{\elem}}{\epsilon} | v |_{H^1(T_F)}^2\Bigr) \\
& = 2Cd \Bigl( \frac{\epsilon}{h_{\elem}} \sum_{F \subset \partial \elem}\| v \|_{L^2(T_F)}^2 + \frac{h_{\elem}}{\epsilon} \sum_{F \subset \partial \elem}| v |_{H^1(T_F)}^2 \Bigr)\\
& \le 2Cd \Bigl( \frac{\epsilon}{h_{\elem}} \| v \|_{L^2(\elem)}^2 + \frac{h_{\elem}}{\epsilon} | v |_{H^1(\elem)}^2 \Bigr),
\end{align} 
where in the last inequality we have used the fact that the simplices of the set $\{T_F: F \subset \partial \elem \}$ satisfy Assumption~\ref{ass1}, in the sense that they are disjoints and  $\cup_{F \subset \partial \elem} \overline{T_F} \subset \overline{\elem}$.
\end{proof}

%%%%%%%%%%%%%%%%%%%%%%%%
%  PROOF APPROXIMATION PROPERTY    %
%%%%%%%%%%%%%%%%%%%%%%%%
\section{Proof of Lemma~\ref{lem:QuasiStability}}\label{appx:app_prop}
In order to show Lemma~\ref{lem:QuasiStability} we follow the analysis presented in \cite{DuanGaoTanZhang}, by firstly showing two preliminary results making use of the properties presented in Sect.~\ref{sec:preliminary}.
\begin{lemma} \label{lem:proprieta1}
Let Assumptions~\ref{ass1} - \ref{ass4} hold, and let $\widetilde{\Pi}_j$ be the projection operator on $V_j$ as defined in Lemma~\ref{lem:interpDG}, for $j=J,J-1$. Then
\begin{equation}
\normL[\widetilde{\Pi}_J w - I_{J-1}^J \widetilde{\Pi}_{J-1} w][2][\Om] \lesssim \frac{h_J^2}{p_J^2}  \normH[w][2][\Om] \quad \forall w \in H^2(\Om).
\end{equation}
\end{lemma}

\begin{proof}
Using the triangular inequality, Remark~\ref{rmrk:eqQJandIJ} and the approximation estimates of Lemma~\ref{lem:interpDG} we have:
\begin{align}
\| \widetilde{\Pi}_J  w -  I_{J-1}^J &\widetilde{\Pi}_{J-1} w  \|_{ L^2(\Om)} \leq \\& \leq \| \widetilde{\Pi}_J w -  w\|_{L^2(\Om)}  + \| w - Q_Jw\|_{L^2(\Om)} + \| Q_J w - I_{J-1}^J \widetilde{\Pi}_{J-1} w\|_{L^2(\Om)}\\
& = \| \widetilde{\Pi}_J w -  w\|_{L^2(\Om)}  + \min_{z_h \in V_J}\|w - z_h\|_{L^2(\Om)} + \|Q_J (w -  \widetilde{\Pi}_{J-1} w)\|_{L^2(\Om)} \\
& \leq \normL[\widetilde{\Pi}_J w -  w][2][\Om]  + \normL[\ w - \widetilde{\Pi}_J w][2][\Om] + \normL[w -  \widetilde{\Pi}_{J-1} w][2][\Om] \\
& \lesssim  \frac{h_J^2}{p_J^2}  \normH[w][2][\Om] +  \frac{h_{J-1}^2}{p_{J-1}^2}  \normH[w][2][\Om] \lesssim \frac{h_J^2}{p_J^2}  \normH[w][2][\Om],
\end{align}
where in the last inequality we also used hypothesis \eqref{eq:hjhj1} and \eqref{eq:pjpj1}.
\end{proof}

\begin{lemma}\label{lem:QuasiQuasiApprox}
Let Assumptions~\ref{ass1} - \ref{ass4} hold. Let be $g \in L^2(\Om)$ and denote by $w_j \in V_j$ the solution of $\Aa[j][w_j][v] = (g,v)_{L^2(\Omega)}$ $\forall v \in V_j$ with $j=J-1,J$. Then the following inequality holds:
\begin{equation}
\normL[w_J - I_{J-1}^Jw_{J-1}][2][\Om] + \normL[w_{J-1} - P_J^{J-1}w_J][2][\Om] \lesssim \frac{h_J^2}{p_J^{2-\mu}} \normL[g][2][\Om].
\label{eq:QQA}
\end{equation}
\end{lemma}

\begin{proof}
Consider the unique solution $w \in V$ of the problem 
\begin{equation}
\mathcal{A}(w,v) = (g,v)_{L^2(\Omega)} \quad \forall v \in V.
\end{equation}
Using Theorem~\ref{thm:errors} we have 
\begin{equation}
\normL[w - w_j][2][\Om] \lesssim \frac{h_j^2}{p_j^{2-\mu}} \normL[w][2][\Om], \quad j = J-1,J.
\label{eq:ProofQQA1}
\end{equation}
Using the triangular inequality and Remark~\ref{rmrk:eqQJandIJ} we have:
\begin{align}
\|w_J - I_{J-1}^Jw_{J-1}&\|_{L^2(\Om)}  \leq \normL[w_J - w][2][\Om] + \normL[w - \widetilde{\Pi}_J w][2][\Om]  \\ 
& \quad + \normL[\widetilde{\Pi}_Jw - I_{J-1}^J \widetilde{\Pi}_{J-1} w][2][\Om] + \normL[I_{J-1}^J \widetilde{\Pi}_{J-1}w - Q_J w][2][\Om] \\ & \quad + \normL[Q_J w - I_{J-1}^Jw_{J-1} ][2][\Om] \\
& = \normL[w_J - w][2][\Om] + \normL[w - \widetilde{\Pi}_J w][2][\Om] + \normL[\widetilde{\Pi}_Jw - I_{J-1}^J \widetilde{\Pi}_{J-1} w][2][\Om] \\ 
& \quad + \normL[Q_J (\widetilde{\Pi}_{J-1}w - w)][2][\Om] + \normL[Q_J(w - w_{J-1}) ][2][\Om]  \\
& \leq \normL[w_J - w][2][\Om] + \normL[w - \widetilde{\Pi}_J w][2][\Om] + \normL[\widetilde{\Pi}_Jw - I_{J-1}^J \widetilde{\Pi}_{J-1} w][2][\Om] \\ 
& \quad + \normL[ \widetilde{\Pi}_{J-1}w - w][2][\Om] + \normL[ w - w_{J-1} ][2][\Om]. 
\end{align}
Using \eqref{eq:ProofQQA1}, Lemma~\ref{lem:interpDG} and Lemma~\ref{lem:proprieta1}, we have
 \begin{align}
\normL[w_J - I_{J-1}^Jw_{J-1}][2][\Om] & \lesssim  \frac{h_J^2}{p_J^{2-\mu}} \normH[w][2][\Om] +  \frac{h_J^2}{p_J^2}  \normH[w][2][\Om] +  \frac{h_J^2}{p_J^2}  \normH[w][2][\Om] \\
& \quad +  \frac{h_{J-1}^2}{p_{J-1}^2}  \normH[w][2][\Om] +  \frac{h_{J-1}^2}{p_{J-1}^{2-\mu}} \normH[w][2][\Om].
\end{align}
From the elliptic regularity assumption \eqref{eq:elliptic_reg} and hypothesis \eqref{eq:hjhj1} and \eqref{eq:pjpj1}, we can write
\begin{equation}\label{eq:ProofQQA2}
\normL[w_J - I_{J-1}^Jw_{J-1}][2][\Om]  \lesssim \frac{h_J^2}{p_J^{2-\mu}}  \normL[g][2][\Om].
\end{equation}
\noindent Now, let $z_j \in V_j$ be the solution of:
\begin{equation}
\Aa[j][z_j][q] = (w_{J-1} - P_J^{J-1} w_J, q_j)_{L^2(\Omega)} \quad \forall q_j \in V_j, \quad j = J-1,J;
\end{equation}
\noindent  Using \eqref{eq:ProofQQA2} we get the following estimate:
\begin{equation}
\normL[z_{J-1} - I_{J-1}^J z_{J-1}][2][\Om] \lesssim \frac{h_J^2}{p_J^{2-\mu}}  \normL[w_{J-1} - P_J^{J-1}w_J][2][\Om].
\end{equation}
\noindent Then, we have:
\begin{align}
\normL[w_{J-1} - P_J^{J-1} w_J][2][\Om]^2  & = \Aa[J-1][z_{J-1}][w_{J-1} - P_J^{J-1} w_J]  \\
& = \Aa[J-1][z_{J-1}][w_{J-1}] - \Aa[J][I_{J-1}^J z_{J-1}][w_J] \\
& = (z_{J-1},g) - (I_{J-1}^J z_{J-1},g) = (g,z_{J-1} - I_{J-1}^J z_{J-1}) \\
& \lesssim \normL[g][2][\Om] \frac{h_J^2}{p_J^{2-\mu}} \normL[w_{J-1} - P_J^{J-1} w_J][2][\Om],
\end{align}
from which, together with \eqref{eq:ProofQQA2}, inequality~\eqref{eq:QQA} follows.
\end{proof}

\begin{proof}[of Lemma~\ref{lem:QuasiStability}]
For any $v_J \in V_J$ we consider the following equality:
\begin{equation}\label{eq:ProofQA1}
\normL[(\textnormal{Id}_J - I_{J-1}^J P_J^{J-1})v_J][2][\Om] = \sup_{0 \ne \phi \in L^2(\Om)}\frac{\bigl( \phi, (\textnormal{Id}_J - I_{J-1}^J P_J^{J-1})v_J \bigl)_{L^2(\Om)}}{\normL[\phi][2]}.
\end{equation}
Next, consider the solution $z_j$ of the following problems
\begin{equation}
\Aa[j][z_j][v_j] = \bigl( \phi, v_j \bigr) \quad \forall v_j \in V_j, \text{ for }j=J,J-1.
\end{equation}
By using the definition of $P_J^{J-1}$ and Lemma~\ref{lem:QuasiQuasiApprox}, we have:
\begin{align}
\bigl(\phi, (\textnormal{Id}_J - I_{J-1}^J & P_J^{J-1}) v_J)  \bigr)_{L^2(\Om)}   = \Aa[J][z_J][v_J] - \Aa[J-1][P_{J}^{J-1}z_J][P_{J}^{J-1}v_J] \\
& = \Aa[J][z_J - I_{J-1}^J z_{J-1}][v_J] + \Aa[J][I_{J-1}^J (z_{J-1} - P_J^{J-1}z_J)][v_J]\\
& \le \ltrivert v_J \rtrivert_{2,J} \Bigl( \normL[z_J - I_{J-1}^J z_{J-1}][2] + \normL[z_{J-1} - P_{J}^{J-1} z_{J}][2] \Bigr) \\
& \lesssim \ltrivert v_J \rtrivert_{2,J} \frac{h_J^2}{p_J^{2-\mu}} \normL[\phi][2].
\end{align}
Using the last inequality together with~\eqref{eq:ProofQA1} we get \eqref{eq:QuasiApprox}.
\end{proof}

%%%%%%%%%%%%%%
%%   BIBLIOGRAPHY  %%
%%%%%%%%%%%%%%

\end{document}